\documentclass{article}
\usepackage[utf8]{inputenc}
\usepackage[english]{babel} 
\usepackage[T1]{fontenc} 
\usepackage[utf8]{inputenc} 
\usepackage{soul}
\usepackage[normalem]{ulem}
\usepackage{fancybox}
\usepackage{fancyvrb}
\usepackage{listings}
\usepackage{moreverb}
\usepackage{url}
\usepackage{graphicx}
\graphicspath{ {./images/} }
\usepackage{textcomp}
\usepackage{lmodern}
\usepackage{eurosym}
\usepackage[cyr]{aeguill}
\usepackage{amsmath} 
\usepackage{wrapfig} 
\usepackage{makeidx} 
\usepackage{picinpar} 
\usepackage[final]{pdfpages} 
\usepackage{array,multirow,tabularx}
\usepackage{amsmath}
\usepackage{amssymb}
\usepackage[a4paper]{geometry}
\usepackage{amsthm}
\usepackage{tikz}
\usepackage{fancyhdr}
\usepackage{colortbl}
\usepackage{color}
\usepackage{setspace}
\usepackage{longtable}
\usepackage{hhline}
\usepackage{arydshln}
\usepackage{makecell}
\usepackage{mathtools}
\usepackage{tkz-tab}
\usepackage{multicol}
\tikzset{node distance=2.0cm, auto}

\newcommand*\quot[2]{{^{\textstyle #1} \big/_{\textstyle #2}}}
\usepackage[section]{placeins}

\DeclareMathOperator{\Hom}{Hom}

\DeclareMathOperator{\C}{\mathbb{C}}

\DeclareMathOperator{\h}{\mathcal{H}}
\DeclareMathOperator{\K}{\mathbb{K}}
\DeclareMathOperator{\R}{\mathbb{R}}

\DeclareMathOperator{\N}{\mathbb{N}}
\DeclareMathOperator{\F}{\mathbb{F}}

\DeclareMathOperator{\Der}{Der}

\DeclareMathOperator{\ad}{ad}
\DeclareMathOperator{\g}{\mathfrak{g}}

\DeclareMathOperator{\im}{\text{Im}}
\DeclareMathOperator{\Ker}{\text{Ker}}
\DeclareMathOperator{\ind}{\text{ind}}
\DeclareMathOperator{\w}{\omega}
\DeclareMathOperator{\id}{\text{id}}
\DeclareMathOperator{\obs}{\text{obs}}
\DeclareMathOperator{\Sl}{\mathfrak{sl}}
\allowdisplaybreaks

\title{Deformations and Cohomology of restricted Lie-Rinehart algebras in positive characteristic}
\author{Quentin Ehret\footnote{Université de Haute-Alsace, IRIMAS UR 7499, F-68100 Mulhouse, France, E-mail: quentin.ehret@uha.fr}  \; and  Abdenacer Makhlouf\footnote{Université de Haute-Alsace, IRIMAS UR 7499, F-68100 Mulhouse, France, E-mail: abdenacer.makhlouf@uha.fr} }
\date{\today}

\begin{document}

\newtheorem{thm}{Theorem}[section]
	\newtheorem{prop}[thm]{Proposition}
	\newtheorem{lem}[thm]{Lemma}
	\newtheorem{cor}[thm]{Corollary}
	\theoremstyle{definition}
	\newtheorem{defi}[thm]{Definition}
	\newtheorem*{rmq}{Remark}

\maketitle
\begin{abstract}
    The main purpose of this paper is  to study  restricted formal deformations of restricted Lie-Rinehart algebras in positive characteristic $p$. For $p>2$, we discuss the deformation theory and show that deformations  are controlled by the restricted cohomology introduced by Evans and Fuchs.  Furthermore, for $p=2$, we introduce a new  cohomology complex   and show that it fits with the  deformation theory  of restricted Lie-Rinehart algebras in characteristic 2. In particular, we study  the structure and cohomology of restricted Heisenberg algebras.
\end{abstract}

\tableofcontents
\section*{Introduction}

Lie-Rinehart algebras are an algebraic version of Lie algebroids. They were first introduced in characteristic 0 by several authors, including Palais (\cite{PR61}) and Rinehart (\cite{RG63}), and were later studied  deeply by Huebschmann (\cite{HJ90}, \cite{HJ98}). A Lie-Rinehart algebra over a commutative ring $\K$ is a pair $(A,L)$, where $A$ is a commutative, associative $\K$-algebra and $L$ is a Lie $\K$-algebra that acts on $A$ as a module through a map called the anchor map, denoted $\rho:L\rightarrow\Der(A)$, where $\Der(A)$ is the Lie algebra of derivations of $A$. This map must be both a Lie and $A$-module map satisfying a compatibility condition.

In positive characteristic, one needs additional structure to deal with Lie-Rinehart algebras. Let $A$ be an associative algebra over a field of positive characteristic $p$. Then the Lie algebra of derivations $\Der(A)$ is equipped with a map $D\mapsto D^p$, which takes values in $\Der(A)$. It is worth noticing that in characteristic 0, this map is generally not a derivation for $D\in \Der(A)$. This led to the definition of restricted Lie algebras, which are Lie algebras equipped with a non-linear maps called  $p$-maps, denoted $(\cdot)^{[p]}: L\rightarrow L$, that satisfy certain compatibility conditions with respect to the Lie bracket and the additive structure of the underlying vector space. These objects were introduced by Jacobson (\cite{JN37}, \cite{JN41}) and later studied by Hochschild (\cite{HG54}, \cite{HG55.1}, \cite{HG55.2}).

The positive characteristic version of Lie-Rinehart algebras, known as restricted Lie-Rinehart algebras, appeared implicitly in the work of Hochschild in the context of Galois-Jacobson theory for purely inseparable extensions of exponent 1 (\cite{HG55.1}). They were later studied by several other authors, see (\cite{RD00}, \cite{DI12}, \cite{SC15}, \cite{SP16}). The deformation of restricted Lie-Rinehart algebras and the associated cohomology are closely related to those of restricted Lie algebras. Hochschild used the notion of restricted universal enveloping algebra to define the restricted cohomology groups of restricted Lie algebras and to establish a connection to the ordinary Chevalley-Eilenberg cohomology of Lie algebras through a six-term exact sequence (\cite{HG54}). May also made contributions to this area (\cite{MJP66}). However, the definition of restricted cohomology given by Hochschild is not well-suited for computation. Evans and Fuchs attempted to address this issue by providing a new free resolution of the ground field in terms of modules over the restricted enveloping algebra, but their approach only allows for complete computations of cohomology up to dimension $p$ in the case where the restricted Lie algebra is abelian. In the non-abelian case, the bracket and non-linearity of the $p$-map make the problem much more difficult, and Evans and Fuchs  only  described the first and second restricted cohomology groups (\cite{ET00}, \cite{EF08}). This is sufficient for certain algebraic interpretations, but not enough to develop a complete theory of formal deformations of restricted Lie algebras in the vein of Gerstenhaber (\cite{GM64}) and Nijenhuis-Richardson (\cite{NR66}, \cite{NR67}).

Many of the previously mentioned results do not hold in characteristic 2. Bouarroudj and his collaborators have done extensive work in this context (\cite{BB18}, \cite{BGL09}, \cite{BLLI15}) and references therein. Recently, Bouarroudj and Makhlouf developed a complete cohomology complex for Hom-Lie superalgebras in characteristic 2, which are generalization of Lie superalgebras in characteristic 2,  and provided several interesting algebraic interpretations, including the control of formal deformations (\cite{BM22}).

The aim of this paper is to study cohomology and deformations of restricted Lie algebras in characteristic $p$ for $p\leq 3$ and $p=2$. It is organized as follows. In the first section, we review the basics of restricted Lie algebras in characteristic $p$, as well as some examples. We then consider cohomology and provide the first terms of the restricted cochain complex and the restricted differentials as defined in \cite{EF08}. In Section 2, we  review the definition of restricted Lie-Rinehart algebras and  introduce the concept of restricted multiderivation for restricted Lie-Rinehart algebras. In Section 3, we  develop a theory of formal deformations and show that it is controlled by the restricted cohomology. Then, in Section 4 we  discuss restricted structures on the Heisenberg algebra of dimension 3. We show that there are (up to restricted isomorphism) three different restricted Heisenberg algebras, and for each of them we compute the second restricted cohomology group with adjoint coefficients. We also construct restricted Lie-Rinehart algebras on these restricted Heisenberg algebras and compute their deformations. Finally, in the last section  focuses  on the specific case of characteristic 2. We construct a new and complete complex of cochains that has no analogue in any characteristic other than 2, provide some algebraic interpretations, and show that this complex controls formal deformations.

Throughout the paper, ``ordinary" shall be understood as ``not restricted" and we assume that algebras are finite-dimensional.


\section{Restricted Lie Algebras}

We first review some basics about  restricted Lie algebras as well as their cohomology.

\subsection{Basics}\label{basics}

Let $\F$ denote a field of characteristic $p\neq 0$. The notions in this section are taken from \cite{SF88} and were first introduced by Jacobson in \cite{JN41}.

\begin{defi}[Restricted Lie Algebra]\label{restdefi}
	A restricted Lie algebra over $\F$ is a Lie $\F$-algebra $L$ endowed with a map $(\cdot)^{[p]}:L\longrightarrow L$ such that
	\begin{enumerate}
		\item $(\lambda x)^{[p]}=\lambda^px^{[p]}$, $x\in L$, $\lambda\in \F$;
		\item $\left[x,y^{[p]} \right]=[ [\cdots[x,  \overset{p\text{ terms}}{\overbrace{y],y],\cdots,y}}]$;
		\item $(x+y)^{[p]}=x^{[p]}+y^{[p]}+\displaystyle\sum_{i=1}^{p-1}s_i(x,y)$,
	\end{enumerate}
 with $is_i(x,y)$ being the coefficient of $Z^{i-1}$ in $\ad^{p-1}_{Z x+y}(x)$. Such a map $(\cdot)^{[p]}:L\longrightarrow L$ is called a $p$-map.
\end{defi}
\noindent We have the explicit formula
\[
 is_i(x,y)=\sum_{\underset{\sharp\{k,~x_k=x\}=i-1}{x_k\in\{x,y\}}}[x_1,[x_2,[\cdots,[x_k,\cdots,[x_{p-2},[y,x]]\cdots], \]
where $\sharp\{k,~x_k=x\}$ refers to the number of $x_k$'s equal to $x$.\\ We refer a restricted Lie algebra by a triple $(L,[ \cdot , \cdot ],(\cdot )^{[p]})$.
\bigskip

\noindent Throughout the paper, we denote by $\sharp\{x\}$ the number of $x$'s among the $x_k$'s. Then we have
\begin{align}\label{si}
\sum_{i=1}^{p-1}s_i(x,y)&=\sum_{i=1}^{p-1}\frac{1}{i}\sum_{\underset{\sharp\{k,~x_k=x\}=i-1}{x_k\in\{x,y\}}}[x_1,[x_2,[\cdots,[x_k,\cdots,[x_{p-2},[y,x]]\cdots]\nonumber\\&=\sum_{\underset{x_{p-1}=y,~x_p=x}{x_k\in\{x,y\}}}\frac{1}{\sharp\{x\}}[x_1,[x_2,[\cdots,[x_k,\cdots,[x_{p-1},x_p]]\cdots],
\end{align}
since $\frac{1}{i}$ is exactly the inverse of the number of $x$'s among the $x_k$'s.

\bigskip

We have the following particular cases:

\begin{itemize}
	\item[$\bullet$] $p=2:$  for $x,y\in L$, we have $\ad_{Z x+y}(x)=[x,Z x+y]=[x,y]$. We deduce that $s_1(x,y)=[x,y]$. Hence $$(x+y)^{[2]}=x^{[2]}+y^{[2]}+[x,y].$$
	
	\item[$\bullet$] $p=3:$ for $x,y\in L$, we have $\ad^2_{Z x+y}(x)=\left[[x,Z x+y],Z x+y \right]=Z[[x,y],x]+[[x,y],y].$ We deduce that $$s_1(x,y)=[[x,y],y];~~ s_2(x,y)=2[[x,y],x].$$
	
\end{itemize}

\begin{rmq}
    If the adjoint representation $\ad:x\mapsto [x,\cdot]$ is faithful (or equivalently, if the Lie algebra is centerless), then both Conditions 1. and 3. in  Definition \ref{restdefi} follow from  Condition 2.
\end{rmq}

\noindent\textbf{Examples:} 
\begin{enumerate}
\item Let $A$ be an associative algebra over $\F$. Endowed  with the bracket $[x,y]=xy-yx$,   $A$ becomes a restricted Lie algebra with the map $x\longmapsto x^p$,
called \textbf{Frobenius morphism}.

\item Let $L$ be an abelian Lie algebra. Then, any map $f: L\rightarrow L$ satisfying $$f(\lambda x+y)=\lambda^p f(x)+f(y),~x,y\in L,~\lambda\in \F$$ is a $p$-map on $L$. A map satisfying such a property is said  \textbf{$p$-semilinear}.

\item Let $\mathbb{F}$ be a field of characteristic $p\geq3$. We consider $\Sl_2(\F)=\text{Span}_{\F} \{X,Y,H\}$ with skew-symmetric brackets $[X,Y]=H,$ $[H,X]=2X,$ $[H,Y]=-2Y$. We endow it with a restricted structure using the $p$-map given by $X^{[p]}=Y^{[p]}=0,$ $H^{[p]}=2^{p-1}H.$ The Lie algebra $\Sl_2(\F)$ is simple, so the restricted structure is unique. Indeed, two $p$-maps on a restricted Lie algebra differ from a map  whose image lies in the center. If the algebra is simple, the center is reduced to $\{0\}$ (see \cite{SF88}).

\begin{rmq}
    In \cite{EM22}, $\Sl_2(\C)$ is denoted $L^6_{3|0}$.
\end{rmq}

\item Let $\mathbb{F}$ be a field of characteristic $p\geq5$. We consider the \textbf{Witt algebra} $W(1)=\text{Span}_{\F} \{e_{-1},e_0,...,e_{p-2}\}$   with the bracket
$$[e_i,e_j]=
    \begin{cases}
		(j-i)e_{i+j} ~\text{ if }~ i+j\in\{-1,...,p-2\};\\
		0 ~\text{ otherwise;}\\
	\end{cases}$$
and the $p$-map
$$e_i^{[p]}=
    \begin{cases}
		e_0^{[p]}=e_0;\\
		e_i^{[p]}=0 ~\text{ if } i\neq 0.\\
	\end{cases}$$

Then $\left(W(1),[\cdot,\cdot], (\cdot)^{[p]}\right)$ is a restricted Lie algebra (see \cite{EFP16}). Moreover, this Lie algebra is also simple, so the restricted structure is unique. 

\noindent One can express the Witt algebra as the derivation algebra of the  commutative associative algebra 
$A:=\quot{\F[x]}{(x^p-1)}$ (see \cite{EF02}). In this framework, the basis elements are $e_i=x^{i+1}\frac{d}{dx}$, the bracket being the commutator: if $f\in A,$ we have $$\left[x^{i+1}\frac{d}{dx},x^{j+1}\frac{d}{dx}\right](f)=(j-i)x^{i+j+1}\frac{d}{dx}(f)\text{ if } i+j+1\in\{-1,...,p-2\} \text{ and } 0 \text{ otherwise. }$$  The $p$-map is then given by $$\left(x\frac{d}{dx}\right)^{[p]}=x\frac{d}{dx} \text{ and } \left(x^k\frac{d}{dx}\right)^{[p]}=0,~k\neq 1.$$


\item  Here again, let $\mathbb{F}$ be a field of characteristic $p\geq5$. We consider the \textbf{restricted filiform Lie algebra} $\mathfrak{m}^{\lambda}_0(p)=\text{Span}\{e_1,\cdots,e_p\}$, with the bracket $[e_1, e_i] = e_{i+1}, ~2\leq i \leq p-1$. Let $\lambda=(\lambda_1,\cdots,\lambda_p)\in\F^p$. Then, it can be shown that $e_k^{[p]}=\lambda_ke_p$ defines a $p$-map on $\mathfrak{m}^{\lambda}_0(p)$ (see \cite{EF19}).

\end{enumerate}

\begin{defi}
	Let $\left( L_1,[\cdot,\cdot]_1,(\cdot)^{[p]_1}\right) $ and $\left( L_2,[\cdot,\cdot]_2,(\cdot)^{[p]_2}\right) $ be two restricted Lie algebras. A \textbf{restricted morphism} (or \textbf{$p$-morphism}) $\varphi:L_1\longrightarrow L_2$ is a Lie morphism that satisfies $\varphi\left(x^{[p]_1} \right)=\varphi(x)^{[p]_2}$. 
\end{defi}

\begin{defi}Let $\left( L,[\cdot,\cdot],(\cdot)^{[p]}\right)$ be a restricted Lie algebra.  A vector space $M$ over $\F$ is an $L$-module  with an action $L\times M\rightarrow M$ if for all $x,y\in L$ and $m\in M$ we have   $[x,y]\cdot m=x\cdot (y\cdot m)-y\cdot (x\cdot m)$. It is called \textbf{restricted} if we have in addition $x^{[p]}\cdot m=\left(\overset{p\text{ terms}}{\overbrace{x\cdot(x\cdots (x}}\cdot m)\cdots)\right)$. 
   
\end{defi}

\begin{thm}[Jacobson's Theorem (\cite{JN62})]\label{jacobson}
	Let $L$ be a $n$-dimensional Lie algebra over a field $\F$ of characteristic $p$. Suppose that $(e_j)_{j\in \{1,\cdots, n\}}$ is a basis of $L$ such that it exists $y_j\in L,~\left(\ad_{e_j} \right)^p=\ad_{y_j}$. Then it exists exactly one $p$-map such that $e_j^{[p]}=y_j,~\forall j=1,\cdots, n$. 
\end{thm}


\subsection{Cohomology of Restricted Lie Algebras}

We assume here that  the ground field $\F$ is of characteristic $p>2$. We recall the definitions of the  Chevalley-Eilenberg cohomology for ordinary Lie algebras (\cite{CE48}) and the restricted cohomology for restricted Lie algebras (\cite{EF08}). In the latter case, we are only able to compute restricted cochains up to order $3$ and restricted cocycles up to order $2$.

\subsubsection{Ordinary Chevalley-Eilenberg Cohomology}

Let $L$ be a Lie algebra and $M$ be a $L$-module. If $m\geq 1$, and $\displaystyle TL=\bigoplus T^mL$ is the tensor algebra of $L$, we have
\[ \Lambda^mL=T^mL/\left\langle x_1\otimes\cdots\otimes x_{k}\otimes x_{k+1}\otimes\cdots \otimes x_m+x_1\otimes\cdots\otimes x_{k+1}\otimes x_k\otimes\cdots \otimes x_m \right\rangle,~~~x_1,\cdots,x_m \in  L.    \]
We define the cochains 
\begin{align*}
C^q_{CE}(L,M)&=\Hom_{\F}(\Lambda^mL, M)~\text{ for }m \geq 1,\\
C^0_{CE}(L,M)&\cong M.
\end{align*}
We define a differential map $d^m_{CE}: C^m_{CE}(L,M)\longrightarrow C^{m+1}_{CE}(L,M)$ by
\begin{align*}
 d_{CE}^m(\varphi)(x_1,\cdots,x_{q+1})&=\sum_{1\leq i<j\leq m+1}(-1)^{i+j-1}\varphi\left([x_i,x_j],x_1,\cdots\hat{x_i},\cdots,\hat{x_j},\cdots, x_{m+1} \right)\\&+\sum_{i=1}^{m+1}(-1)^{i}x_i\varphi(x_1,\cdots,\hat{x_i},\cdots,x_{m+1}).   
 \end{align*}
 
We have $d^{m+1}_{CE}\circ d_{CE}^{m}=0$. As usual we denote  the $m$-cocycles by $Z^m_{CE}(L,M)=\Ker(d_{CE}^m)$ and $m$-coboundaries by $B^m_{CE}(L,M)=\im(d_{CE}^{m-1})$.
Then we define the ordinary Chevalley-Eilenberg cohomology of $L$ with values in $M$ by
$$  H^m_{CE}(L,M)=Z^m_{CE}(L,M)/B^m_{CE}(L,M).    $$

\begin{prop}
    Let $L$ be a Lie algebra over a field of characteristic $p>0$ and suppose that $H_{CE}^1(L,L)=0$. Then,    $L$ admits  a $p$-map.
\end{prop}
\begin{proof}
    Since $H_{CE}^1(L,L)=0$, every derivation is inner. In particular, if $\{e_1,\cdots,e_n\}$ is a basis of $L$,  the derivation $(\ad_{e_i})^p,~i=1,\cdots,n$ is inner. Therefore, it exists $u_i\in L,~i=1,\cdots,n$ such that $(\ad_{e_i})^p=\ad_{u_i}$. The  conclusion follows thanks to Jacobson's Theorem \ref{jacobson}.
\end{proof}

\subsubsection{Restricted cohomology}

Let $L$ be a restricted Lie algebra and $M$ be a  restricted $L$-module.

\begin{defi}
	Let $\varphi\in C^2_{CE}(L,M)$ and $\omega: L\longrightarrow M$ be a map.
	We say that $\omega$ has the $(*)$-property with respect to $\varphi$ if
	\begin{align}
		 \omega(\lambda x)&=\lambda^p\omega(x),~\forall \lambda\in \F,~\forall x\in L;\\
		 \omega(x+y)&=\omega(x)+\omega(y)+\displaystyle\sum_{\underset{x_1=x,~x_2=y}{x_i\in\{x,y\}}}\frac{1}{\sharp\{x\}}\sum_{k=0}^{p-2}(-1)^kx_p\cdots x_{p-k+1}\varphi([[\cdots[x_1,x_2],x_3]\cdots,x_{p-k-1}],x_{p-k}),
	\end{align}
with $x,y\in L$ and $\sharp\{x\}$ is the number of factors $x_i$ equal to $x$.
We set 

$$ C^2_*(L,M)=\left\lbrace (\varphi,\omega),~\varphi\in C^2_{CE}(L,M),~\omega: L\longrightarrow M \text{ has the $(*)$-property w.r.t } \varphi \right\rbrace.     $$
\end{defi}

\noindent\textbf{Example:} Let $p=3$ and  $\varphi\in C^2_{CE}(L,M)$. Then a map $\omega:L\rightarrow M$ has the ($*$)-property with respect to $\varphi$ if and only if
\begin{align}
    \omega(\lambda x)&=\lambda^3\w(x),~\forall\lambda\in \F,~\forall x\in L;\\
    \w(x+y)&=\w(x)+\w(y)+\varphi([x,y],y)+\frac{1}{2}\varphi([x,y],x)-\frac{1}{2}x\cdot\varphi(x,y)-y\cdot\varphi(x,y),~\forall x,y\in L.\label{pssh}
\end{align}
Equation (\ref{pssh}) can be rewritten as
$$ \w(x+y)=\w(x)+\w(y)+\varphi([x,y],y)-\varphi([x,y],x)+x\cdot\varphi(x,y)-y\cdot\varphi(x,y),~\forall x,y\in L.$$

\begin{defi}
	Let $\alpha\in C^3_{CE}(L,M)$ and $\beta: L\times L\longrightarrow M$ be a map. We say that $\beta$ has the $(**)$-property with respect to $\alpha$ if
	\begin{enumerate}
		\item $\beta(x,y)$ is linear with respect to $x$;
		\item $\beta(x,\lambda y)=\lambda^p\beta(x,y)$;
		\item \begin{align*}\beta (x,y_1+y_2)&=\beta(x,y_1)+\beta(x,y_2)\\&-\displaystyle\sum_{\underset{h_1=y_1,~h_2=y_2}{h_i\in\{y_1,y_2\}}}\frac{1}{\sharp\{y_1\}}\sum_{j=0}^{p-2}(-1)^j\\&\times\sum_{k=1}^{j}\binom{j}{k}h_p\cdots h_{p-k-1}\alpha\left( [\cdots[x,h_{p-k}],\cdots,h_{p-j+1}],[\cdots[h_1,h_2],\cdots,h_{p-j-1}],h_{p-j} \right),\end{align*} 
	\end{enumerate}
with $\lambda\in \F$, $x,y,y_1,y_2\in L$ and $\sharp\{y_1\}$ the number of factors $h_i$ equal to $y_1$. We set 
$$ C^3_*(L,M)=\left\lbrace (\alpha,\beta),~\alpha\in C^3_{CE}(L,M),~\beta: L\times L\longrightarrow M \text{ has the $(**)$-property w.r.t } \alpha \right\rbrace.     $$

\end{defi}



\noindent In order to define the differentials, we observe that an element $\varphi\in C^1_*(L,M)$ induces a map
\begin{align*}
\ind^1(\varphi): L&\longrightarrow M\\
x&\longmapsto \varphi\left(x^{[p]}\right)-x^{p-1}\varphi(x). 
\end{align*}

\noindent We also observe that an element $(\alpha,\beta)\in C_*^2(L,M)$ induces a map
\begin{align*}
	\ind^2(\alpha,\beta):L\times L&\longrightarrow M \\
	(x,y)&\longmapsto \alpha\left( x,y^{[p]}\right) -\displaystyle\sum_{i+j=p-1}(-1)^iy^i \alpha\left( [\cdots[x,\overset{j\text{ terms}}{\overbrace{y],\cdots,y}}],y\right)+x\beta(y).
	\end{align*}

\begin{lem}[\text{\cite{EF08}}]
The map $\ind^1(\varphi)$ satisfies the $(*)$-property with respect to $d^1_{CE}\varphi$, and the map $\ind^2(\alpha,\beta)$ satisfies the $(**)$-property with respect to $d_{CE}^2\alpha$.
\end{lem}

\begin{defi}
	The restricted differentials are defined as follows:
	\begin{align*}
		d_*^0:~&C_*^0(L,M)\longrightarrow C^1_*(L,M),~d_*^0=d_{CE}^0;\\
		d_*^1:~&C_*^1(L,M)\longrightarrow C^2_*(L,M),~d_*^1(\varphi)=\left( d_{CE}^1\varphi,\ind^1(\varphi)\right);\\
		d_*^2:~&C_*^2(L,M)\longrightarrow C^3_*(L,M),~d_*^2(\alpha,\beta)=\left( d_{CE}^2\alpha,\ind^2(\alpha,\beta)\right).
	\end{align*}
\end{defi}
If $1\leq m\leq 2$, we have $d_*^m\circ d^{m-1}_*=0$. We denote by $Z_*^m(L,M)=\Ker(d^m_*)$ the restricted $m$-cocycles and $B_*^m(L,M)=\im(d^{m-1}_*)$ the restricted $m$-coboundaries. Finally we denote the \textbf{restricted cohomology groups} by
$$ H_*^m(L,M)=Z_*^m(L,M)/B^m_*(L,M).        $$

\noindent\textbf{Remark:} $H^0_*(L,M)=H^0_{CE}(L,M)$.

\subsubsection{The case of restricted abelian Lie algebras}

In this section, let $\F$ be a field of characteristic $p>3$ and $L$ be a restricted abelian Lie algebra. In this  case, any $p$-semilinear map $f:L\rightarrow L$ gives a restricted structure on $L$. In \cite{EF08}, the authors  constructed a restricted cohomology up to order $p$ for the abelian Lie algebra. The  complex is defined as follows.

\begin{defi}
	Let $V$ be an $\F$-vector space. For $\lambda\in \F$, we consider  the Frobenius map $F:\lambda\longmapsto \lambda^p$. If the $\F$-vector space structure on $V$ is given by $\F\longrightarrow \text{End}(V)$, we denote by $\bar{V}$ the $\F$-vector space (isomorphic to $V$) which is given by pre-composing $\F\longrightarrow \text{End}(V)$ by $F^{-1}$.
\end{defi}

Let $k\leq p$,  $L$ be a restricted abelian Lie $\F$-algebra and $M$ be a restricted $L$-module. We define
$$ C_{ab}^k(L,M)=\displaystyle\bigoplus_{2t+s=k}      \Hom_{\F}\left(S^t\bar{L}\otimes\wedge^s L,M\right), $$ where $S\displaystyle\bar{L}=\bigoplus_{t\geq0}S^t\bar{L}$ is the symmetric algebra on $\bar{L}$. An element $\gamma\in C_{ab}^k(L,M)$ is a family $\gamma=\left\lbrace\gamma_t \right\rbrace_{0\leq t\leq \lfloor\frac{k}{2}\rfloor},$ where $\lfloor\cdot \rfloor$ denotes the floor function, with
\begin{align*}
\gamma_t:S^t\bar{L}\otimes\wedge^s L&\longrightarrow M,~~(2t+s=k)\\
(x_1,..., x_t, y_1,..., y_s)&\longmapsto \gamma_t(x_1,..., x_t, y_1,..., y_s).
\end{align*}

Those maps $\gamma_t$ are linear skew-symmetric in the $y$'s and $p$-semilinear symmetric in the $x$'s.

Then, we define a differential operator on $C_{ab}^k(L,M)$ by
\begin{align*}
d^k_{ab}:C_{ab}^k(L,M)&\longrightarrow C_{ab}^{k+1}(L,M)\\
\left\lbrace\gamma_t \right\rbrace_{0\leq t\leq \lfloor\frac{k}{2}\rfloor}&\longmapsto \left\lbrace\beta_t \right\rbrace_{0\leq t\leq \lfloor\frac{k+1}{2}\rfloor},
\end{align*}
with
\begin{align*}
\beta_t(x_1,..., x_t; y_1,..., y_s)&=\displaystyle\sum_{j=1}^{s}(-1)^j y_j\cdot\gamma_t(x_1,..., x_t, y_1,...,\hat{y_j},..., y_s)\\
&+\displaystyle\sum_{i=1}^{t}\gamma_{t-1}(x_1,...,\hat{x_i},..., x_t, x_i^{[p]}, y_1,..., y_s)\\
&+\displaystyle\sum_{i=1}^{t}x_i^{p-1}\cdot \gamma_{t-1}(x_1,...,\hat{x_i},..., x_t, x_i, y_1,..., y_s).
\end{align*}

\begin{prop}[\text{\cite{EF08}}]
   The complex  $\left(C_{ab}^k(L,M),d^k_{ab}\right)_{0\leq k\leq p}$ is a  cochains complex.
\end{prop}

We sum this up by drawing the following diagram:
{\tiny
\begin{center}
	\begin{tikzpicture}
	\node (00) at (1,0) {$0$};
	\node (01) at (1.5,0) {$M$};
	\node (02) at (3,0) {$\Hom(L,M)$};
	\node (03) at (5.5,0) {$\Hom(\wedge^2L,M)$};
	\node (04) at (8.5,0) {$\Hom(\wedge^3L,M)$};
	\node (05) at (12.5,0) {$\Hom(\wedge^4L,M)$};
	\node (06) at (16,0) {$\Hom(\wedge^5L,M)~...$};
	\node (13) at (5.5,-1) {$\Hom(\bar{L},M)$};
	\node (14) at (8.5,-1) {$\Hom(\bar{L}\otimes L,M)$};
	\node (15) at (12.5,-1) {$\Hom(\bar{L}\otimes \wedge^2L,M)$};
	\node (16) at (16,-1) {$\Hom(\bar{L}\otimes \wedge^3L,M)~...$};
	\node (25) at (12.5,-2) {$\Hom(S^2\bar{L},M)$};
	\node (26) at (16,-2) {$\Hom(S^2\bar{L}\otimes L,M)~...$};
	\draw[->]  (00) to (01);
	\draw[->]  (01) to (02);
	\draw[->]  (02) to (03);
	\draw[->]  (03) to (04);
	\draw[->]  (04) to (05);
	\draw[->]  (05) to (06);
	\draw[->]  (13) to (14);
	\draw[->]  (14) to (15);
	\draw[->]  (15) to (16);
	\draw[->]  (02) to (13);
	\draw[->]  (03) to (14);
	\draw[->]  (04) to (15);
	\draw[->]  (05) to (16);
	\draw[->]  (25) to (26);
	\draw[->]  (14) to (25);
	\draw[->]  (15) to (26);
	\node (3) at (5.5,-0.5) {$\bigoplus$};
	\node (4) at (8.5,-0.5) {$\bigoplus$};
	\node (5) at (12.5,-0.5) {$\bigoplus$};
	\node (6) at (16,-0.5) {$\bigoplus$};
	\node (7) at (16,-1.5) {$\bigoplus$};
	\node (7) at (12.5,-1.5) {$\bigoplus$};
	\end{tikzpicture}
\end{center}}

On the first line, we recognize the usual Chevalley-Eilenberg complex.

\section{Restricted Lie-Rinehart Algebras}
Lie-Rinehart algebras over characteristic zero fields appear as algebraic analogs of Lie algebroids in differential geometry. Palais (\cite{PR61}) and Rinehart (\cite{RG63}) were among the first authors to consider those structures. For example, if $V$ is a differential manifold, let $A = \mathcal{O}_V$ be the algebra of smooth functions on $V$, and
$L= Vect(V)$ be the Lie algebra of vector fields on $V$. Then the pair $(A,L)$ carries a
Lie-Rinehart algebra structure. In positive characteristic, the Lie algebra is required to be
restricted and additional structure arise through interactions between the structure maps of the Lie-Rinehart algebra and the $p$-map of the restricted Lie algebra.
\subsection{Definitions}
In the following definition, $\F$ is an arbitrary field. Let $A$ be an $\F$-associative algebra, the space of derivations of $A$ is defined by
$$\Der(A)=\{D:A\rightarrow A \text{ linear map such that } D(ab)=D(a)b+aD(b),~a,b\in A \}.  $$
\begin{defi}\label{superdef}
	A \textbf{Lie-Rinehart algebra} is a pair $(A,L)$, where
	\begin{itemize}
		\item [$\bullet$] $L$ is a Lie algebra over $\F$, with bracket $[\cdot,\cdot]$;
		\item [$\bullet$] $A$ is an  associative and commutative $\F$-algebra,
	\end{itemize}
such that, for $x,y\in L$ and $a,b\in A$:
	\begin{itemize}
		\item There is an action $A\times L\longrightarrow L,~(a,x)\longmapsto a\cdot x,$ making $L$ an $A$-module;
		\item There is a map $L\longrightarrow \text{Der}(A),~
		x\longmapsto \left( \rho_x:a\longmapsto \rho_x(a)\right)$, which is both a morphism of Lie algebras and a morphism of A-modules, called \textbf{anchor map};
		\item There is a compatibility condition, also called \textbf{Leibniz condition}: $[x,a\cdot y]=\rho_x(a)\cdot y+a\cdot [x,y].$
	
	\end{itemize}
\end{defi}

\noindent\textbf{Example: }Let $A$ be an associative unital algebra, and $L=\Der(A)$ be its algebra of derivations. Then one can check that the pair $\left( A, \Der(A)\right)$ is a Lie-Rinehart algebra, with the action $A\curvearrowright \Der(A)$ being given by $(a\cdot D)(x)=aD(x)$ and the trivial anchor given by $\rho(D)=D,~\forall D \in \Der(A), ~a\in A,~x\in L$.

\vspace{0.3cm}

\noindent\textbf{Remark:} A computer-aided classification of Lie-Rinehart (super)algebras over the complex field $\C$ has been carried in \cite{EM22}.

\vspace{0.3cm}

Let $A$ be a  commutative associative algebra. Then it is well-known that the pair $\left(A, \Der(A) \right) $ can be endowed with a Lie-Rinehart structure. If the ground field $\F$ is of prime characteristic $p$, $\Der(A)$ has a richer structure than just a Lie algebra (\cite{DI12}). Let $D\in \Der(A)$. For $a,b\in A$, we have
	
	\begin{equation}\label{derip}  D^p(ab)=\sum_{i=0}^{p}\binom{p}{i}D^i(a)D^{p-i}(b)=aD^p(b)+D^p(a)b.   \end{equation}
	We conclude that $D^p$ is also a derivation. It is then possible to endow $\Der(A)$ with a restricted Lie algebra structure, with the $p$-map $D\longmapsto D^p$. We now recall the definition of a representation of a Lie-Rinehart algebra.

\begin{defi}[\text{\cite{SP16}}]
	Let $(A,L)$ be a Lie-Rinehart algebra, with an anchor map $\rho$. A representation of $(A,L)$ is an $A$-module $M$ endowed with a left $A$-linear Lie algebra map $\pi:L\longrightarrow \text{End}(M), ~x\longmapsto\pi_x$ such that
	\[ \pi_x(am)=a\pi_x(m)+\rho_x(a)m,~~~ a\in A,~m\in M,~ x\in L . \]
\end{defi}

The following lemma is a reformulation of Lemma 1 of \cite{HG55.1}.
\begin{lem}[\text{\cite{SP16}}]
	Let $M$ be a representation of a Lie-Rinehart algebra $(A,L)$, with an anchor map $\rho$. Then the following relation holds in $\text{End}(M),$ for $a\in A,~x\in L:$
	
	\[ \pi_{ax}^p=a^p\pi_x^p+\rho_{ax}^{p-1}(a)\pi_x.   \] 
\end{lem}
\noindent This lemma gives us the following relation between $D\in \Der(A)$ and $a\in A$, $\rho$ being the identity:
 
 \[ (aD)^p=a^pD^p+(aD)^{p-1}(a)D.   \]
 
 This lead to the following general definition of a restricted Lie-Rinehart algebra, which appears implicitly in \cite{HG55.1} and explicitly in \cite{DI12} and in \cite{RD00}.
 
 \begin{defi}
A restricted Lie-Rinehart algebra over a field $\F$ of characteristic $p$ is a pair $(A,L)$ with $A$ a  commutative associative $\F$-algebra  and $\left(L, (-)^{[p]}\right)$ be a restricted Lie $\F$-algebra satisfying the following conditions:
 	\begin{itemize}
 		\item $(A,L)$ is a Lie-Rinehart algebra, with anchor map $\rho:L\longrightarrow\Der(A)$;
 		\item $\rho(x^{[p]})=\rho(x)^p$ ($\rho$ is a restricted Lie morphism);
 		\item $(ax)^{[p]}=a^px^{[p]}+\rho(ax)^{p-1}(a)x,~a\in A,~x\in L$ (Hochschild condition).
 	\end{itemize}
 \end{defi}
 
 \noindent\textbf{Examples:}
 \begin{enumerate}
 	\item If $A$ is a  commutative associative $\F$-algebra, with unit and $\left(L,(\cdot)^{[p])}\right)$ is a restricted Lie algebra, we can always endow the pair $(A,L)$ with a restricted Lie-Rinehart structure with the trivial action and the null anchor (\cite{EM22}).
 	\item We have seen above that $(A,\Der(A))$ can be endowed with a restricted Lie-Rinehart structure, with $A$ a  commutative associative $\F$-algebra.
 	\item Let $(L, (-)^{[p]})$ be a restricted Lie algebra over $\F$. Then $(\F,L)$ can be endowed with a restricted Lie-Rinehart structure.
 	\item Restricted Lie-Rinehart algebras appear in the Galois-Jacobson theory of purely inseparable extensions of exponent 1, see \cite{DI12} for some references.
 	\item Let $(\g,[\cdot,\cdot ]_{\g})$ be a restricted Lie algebra, endowed with a Lie morphism $\gamma:\g\longrightarrow \Der(A)$, where $A$ is an associative algebra. Then $L:=A\otimes\g$ have a restricted Lie-Rinehart structure with ($a,b\in A,~x,y\in L$):
    \begin{align*}
 	[a\otimes x, b\otimes y]_L&=ab\otimes[x,y]+a\gamma(x)(b)\otimes y-b\gamma(y)(a)\otimes x;\\ 
 	\rho_L(a\otimes x)(b)&=a\gamma(x)(b);\\   
 	(a\otimes x)^{[p]_L}&=a^p\otimes x^{[p]_{\g}}-\left(a\gamma(x) \right)^{p-1}(a)\otimes x. 
 	\end{align*}

 	\item We consider the Witt algebra  $W(1)$ described in the Section \ref{basics}. We have seen that this restricted Lie algebra can be realized as $W(1)=\Der(A),$ with $A=\F[x]/(x^p-1)$. As a consequence, $\left(A,W(1)\right)$ is endowed with a restricted Lie-Rinehart structure, with anchor map $\rho=\id$.
 \end{enumerate}
 
\begin{defi}
	Let $(A,L)$ be a restricted Lie-Rinehart algebra. A representation $(\pi,M)$ of $(A,L)$ is said to be restricted if $\pi(x^{[p]})=\pi(x)^p,~\forall x\in L$.
\end{defi} 

\begin{prop}[\text{\cite{DI12}}, Jacobson's Theorem for restricted Lie-Rinehart algebras]
Let $(A,L,\rho)$ be a Lie-Rinehart algebra such that $L$ is free as an  $A$-module. Let $(u_i)_i$ be an $A$-basis of $L$. If there exists a map $u_i\mapsto u_i^{[p]}$ such that $\ad_{u_i}^p=\ad_{u_i^{[p]}}~\forall i$, then $(A,L,\rho)$ can be endowed with a restricted Lie-Rinehart structure. 

\end{prop}

\subsection{Restricted multiderivations and Lie-Rinehart structures}

Let $A$ be an associative algebra and $L$ be a restricted Lie algebra such that $A$ acts on $L$.

\begin{defi}
Let $V$ be an $\F$-vector space. A map $\varphi: V\rightarrow V$ is called \textbf{$p$-homogeneous} if it satisfies $\varphi(\lambda x)=\lambda^p\varphi(x)$, for all $x\in L$ and $\lambda\in \F$.
\end{defi}

\begin{defi}\label{restmulti}
A \textbf{restricted multiderivation} (of order $1$) is a pair $(m,\w)$, where $m:L\times L\rightarrow L$ is a skew-symmetric bilinear map, $\w:L\rightarrow L$ is a $p$-homogeneous map  satisfying
\begin{equation}
    \w(x+y)=\w(x)+\w(y)+\displaystyle\sum_{i=1}^{p-1}\theta_i(x,y),
\end{equation}
where $i\theta_i(x,y)$ is the coefficient of $Z^{i-1}$ in $\left(\tilde{\ad}_m(Z x+y)\right)^{p-1}(x)$, with $
\tilde{\ad}_m(x)(y):=m(x,y)$, such that it exists a map $\sigma:L\rightarrow \Der(A)$ called \textbf{restricted symbol map} which must satisfy the following four conditions, for all $x,y\in L$ and $a\in A$:
\begin{equation}\label{jesaispas1}
    \sigma(ax)=a\sigma(x);
\end{equation}
\begin{equation}\label{jesaispas2}
    m(x,ay)=am(x,y)+\sigma(x)(a)y;
\end{equation}
\begin{equation}\label{jesaispas3}
    \sigma\circ\w(x)=\sigma(x)^p;
\end{equation}
\begin{equation}\label{jesaispas4}
    \w(ax)=a^p\w(x)+\sigma(ax)^{p-1}(a)x.
\end{equation}
\end{defi}

\begin{rmq}
    If Conditions (\ref{jesaispas3}) and (\ref{jesaispas4}) are not satisfied, then we say that $m$ (without $\w$) is an ordinary multiderivation (see \cite{MM20} and \cite{EM22}).
\end{rmq}

\begin{prop}\label{ontoonep}
There is a one-to-one correspondence between  restricted Lie-Rinehart structures on the pair $(A,L)$ and restricted multiderivations $(m,\w )$ of order $1$ satisfying 
\begin{equation}\label{jesaispas5}
    m(x,m(y,z))+m(y,m(z,x))+m(z,m(x,y))=0
\end{equation}
and
\begin{equation}\label{jesaispas6}
    m(x,\w(y))=m(m(\cdots m(x,\overset{p\text{ terms}}{\overbrace{y),y),\cdots,y}})
\end{equation}
\end{prop}

\begin{proof}
    Suppose that $\left(A,L,[\cdot,\cdot], (\cdot)^{[p]}, \rho\right)$ is a restricted Lie-Rinehart algebra. We set $m:=[\cdot,\cdot]$, $\w:=(\cdot)^{[p]}$ and $\sigma:=\rho$. Then, Equation (\ref{jesaispas5}) is nothing more than the Jacobi identity of the Lie bracket and Equation (\ref{jesaispas6}) is satisfied as well, due to the definition of the $p$-map on $L$. Moreover, it is easy to see that $\sigma=\rho$ is a suitable symbol map with respect to $(m,\w)$, so the pair $(m,\w)$ is a $1$-order restricted multiderivation with symbol map $\sigma$ defined above.
    Conversely, if $(m,\w)$ is a $1$-order restricted multiderivation with symbol map $\sigma$, it is easy to see that $\left(A,L,m,\w,\sigma\right)$ is a restricted Lie-Rinehart algebra.

\end{proof}

\begin{rmq}
    If Condition (\ref{jesaispas6}) is not satisfied, then there is a one-to-one correspondence between (ordinary) multiderivations $m$ of order $1$ satisfying Equation (\ref{jesaispas5}) and Lie-Rinehart algebras on $(A,L)$ (see \cite{MM20} and \cite{EM22}). 
\end{rmq}

\section{Deformation theory of restricted Lie-Rinehart algebras}

Let $\F$ be a field of characteristic $p>2$. In \cite{EF08}, Evans and Fuchs have sketched a deformation theory of restricted Lie algebras. They introduced the notions of infinitesimal restricted deformations and  equivalence of such deformations, in order to prove that the second restricted cohomology group classifies the infinitesimal restricted deformations up to equivalence (Theorem 6 of \cite{EF08}). Here we investigate formal restricted deformations of restricted Lie-Rinehart algebras. Proposition \ref{ontoonep} allows us to consider the deformations of restricted multiderivations of order $1$ to study deformations of restricted Lie-Rinehart algebras. We follow to this end, the Gerstenhaber's approach, where the base field $\F$ is replaced by the formal power series ring  $\F[[t]]=\{ \sum_{i\geq 0} a_it^i, a_i\in \F\}$.

\subsection{Restricted formal deformations}

Let $\left(A, L, [\cdot,\cdot],(-)^{[p]}, \rho \right)$ be a restricted Lie-Rinehart algebra. We aim to deform the Lie bracket, the $p$-map and the anchor map. Let $(m,\w)$ be the associated restricted multiderivation.

\begin{defi}\label{defidefop}
	A formal deformation of $(m,\w)$ is given, for all  $x,y\in L$, by two maps
	
	$$m_t: (x,y)\longmapsto \displaystyle\sum_{i\geq 0} t^i m_i(x,y),~~\w_t:x\longmapsto \displaystyle\sum_{j\geq 0}t^j\w_j(x),$$
with $m_0=m$, $\w_0=\w$, and $(m_i,\w_i)$ are restricted multiderivations.
Moreover, the four following conditions must be satisfied, for all $x,y,z\in L,$ and $a\in A$:
	
	\begin{equation}\label{jacomulti}
		 m_t(x,m_t(y,z))_t+m_t(y,m_t(z,x))+m_t(z,m_t(x,y))=0;
		 \end{equation}
		 \begin{equation}\label{pmapmulti}
		 m_t\left(x,\w_t(y)\right)_t=m_t( m_t(\cdots m_t(x,  \overset{p\text{ terms}}{\overbrace{y),y),\cdots,y}});
	\end{equation}
		\begin{equation}\label{condenplus1}
	    \sum_{i=0}^k\sigma_i\left(\w_{k-i}(x)\right)(a)=\sum_{i_1+\cdots+i_p=k}\sigma_{i_1}(x)\circ\cdots\circ\sigma_{i_p}(x)(a),~~\forall k\geq 0;
	\end{equation}
	\begin{equation}\label{condenplus2}
	    \sigma_k(x)^{p-1}=\sum_{i_1+\cdots+i_{p-1}=k}\sigma_{i_1}(x)\circ\sigma_{i_2}(x)\circ \cdots \sigma_{i_{p-1}}(x)~~\forall k\geq 0;
	\end{equation}
\end{defi}

\begin{rmq}
\hspace{0.2cm}
	\begin{enumerate}
	\item $m_t$ extends to the formal space  $L[[t]]=\{ \sum_{i\geq 0}  t^ix_i, x_i\in L\}$ by $\F[[t]]$-linearity.
	\item $\w_t$ extends to $L[[t]]$ by $p$-homogeneity and by using the formula $$\w_t(x+ty)=\w_t+\w_t(ty)+\displaystyle\sum_{k=1}^{p-1}\tilde{s}(x,ty),$$ with $k\tilde{s}(x,ty)$ being the coefficient of $Z^{p-1}$ in the formal expression $m_t(xZ+ty,x)$.
\end{enumerate}
\end{rmq}

\begin{rmq}\label{precisionsp}
	Condition $(\ref{jacomulti})$ ensures that the deformed object $\left(A, L[[t]], m_t,\sigma_t\right)$ is a Lie-Rinehart algebra. Conditions $(\ref{jacomulti})$ and $(\ref{pmapmulti})$ ensure that $\left( L[[t]], m_t,\w_t\right)$ is a restricted Lie algebra. Moreover, if the conditions $(\ref{condenplus1})$ and $(\ref{condenplus2})$ are satisfied, then $\left(A, L[[t]], m_t,\w_t,\sigma_t\right)$ is a restricted Lie-Rinehart algebra.
    Indeed, suppose that Equation (\ref{jacomulti}) is satisfied. It is straightforward to verify that $m_t$ is a multiderivation with symbol map $\sigma_t$. By Proposition 4.3 of \cite{EM22}, $\left(A, L[[t]], m_t,\sigma_t\right)$ is a Lie-Rinehart algebra. If, moreover,  Equation (\ref{pmapmulti}) is satisfied, then $\w_t$ is a $p$-map with respect to $m_t$. As a consequence, $\left( L[[t]], m_t,\w_t\right)$ is a restricted Lie algebra.
    Now, suppose that the four equations (\ref{jacomulti}), (\ref{pmapmulti}), (\ref{condenplus1}) and (\ref{condenplus2}) are satisfied. We already know that $\left(A, L[[t]], m_t,\sigma_t\right)$ is a Lie-Rinehart algebra and that $\left( L[[t]], m_t,\w_t\right)$ is a restricted Lie algebra. It remains to verify that $\w_t$ and $\sigma_t$ satisfy Equations (\ref{jesaispas3}) and (\ref{jesaispas4}) of the Definition \ref{restmulti}. Suppose that $(\ref{condenplus1})$ holds. We have
    
    $$\sum_{k\geq 0}t^k\sum_{i=0}^k\sigma_i\left(\w_{k-i}(x)\right)(a)=\sum_{k\geq 0}t^k\sum_{i_1+\cdots+i_p=k}\sigma_{i_1}(x)\circ\cdots\circ\sigma_{i_p}(x)(a),$$ which is equivalent to
    $$\sum_{i,j\geq 0}t^k\sum_{i=0}^k\sigma_i\left(\w_{j}(x)\right)(a)=\left(\sum_{k\geq0}\sigma_k(x)\right)^p(a).$$ Thus,
    $$\sigma_t\left(\w_t(x)\right)(a)=\sigma_t(x)^p(a).$$
    Condition $(\ref{condenplus2})$ is obtained in a similar way.
    \end{rmq}

\begin{defi}
    If Conditions $(\ref{condenplus1})$ and $(\ref{condenplus2})$  of  Definition $\ref{defidefop}$ are not satisfied, the deformation is called \textbf{weak deformation}.
\end{defi}

\noindent\textbf{Notation.} If $f$ is a map taking two variables, we will use the notation $$f[x_1,\cdots,x_n]:=f(f(\cdots f(x_1,x_2),x_3),\cdots,x_n).$$

\begin{rmq}
    All of the following results concerning restricted deformations, equivalences and obstructions are also valid for restricted Lie algebras, by forgetting the Lie-Rinehart structure. Weak deformations of a restricted  Lie-Rinehart algebra correspond to deformations of  restricted Lie algebras.
\end{rmq}

\begin{lem}
	Let $\left(m_t, \w_t \right)$ be a restricted deformation of $\left(m,\w \right)$. Then $\w_1$ has the $(*)$-property with respect to $m_1$.
\end{lem}
	
\begin{proof}
	Let $\lambda\in \F,~x\in L$.
	\begin{enumerate}
		\item[$\bullet$] $\w_t(\lambda x)=\lambda^p \w_t(x)\Leftrightarrow \w(\lambda x)+t\w_1(\lambda x)+\displaystyle\sum_{i\geq 2}t^i\w_i(\lambda x)=\lambda^p\left(\w(x)+t\w_1(x)+ \displaystyle\sum_{i\geq 2}t^i\w_i(x)\right). $ By collecting the coefficients of $t$ in the above equation, we obtain $\w_1(\lambda x)=\lambda^p\w_1(x)$.
		
		\item[$\bullet$] The following computations are made modulo $t^2$. We have
		\begin{equation}
		m_t[x_1,\cdots,x_p]=m[x_1,\cdots,x_p]+t\left(\sum_{k=0}^{p-2}m\left[m_1(m[x_1,\cdots,x_{p-k-1}],x_{p-k}),x_{p-k+1},\cdots,x_p \right]  \right). 
		\end{equation}
		
		We denote by $(\maltese)$ the conditions $( x_i\in\{x,y\},~x_1=x,~x_2=y )$.   Using the definition of $\w_t$ and Equation (\ref{si}), we have
		\begin{align*}
		\hspace{-1cm}\w_t(x+y)&=\w_t(x)+\w_t(y)+\sum_{(\maltese)} \frac{1}{\sharp\{x\}}m_t\left[x_1,x_2,\cdots,x_p \right]\\
		&=\w(x)+\w(y)+\sum_{(\maltese)}\frac{1}{\sharp\{x\}}m\left[x_1,x_2,\cdots,x_p \right]\\&\hspace{1cm}+t\sum_{(\maltese)}\left( \sum_{k=0}^{p-2}m\left[m_1(m[x_1,\cdots,x_{p-k-1}],x_{p-k}),x_{p-k+1},\cdots,x_p\right]             \right)\mod(t^2). 
		\end{align*}
		
We also have  $\w_t(x+y)=\w(x+y)+t\w_1(x+y)\mod(t^2).$ If we compare the coefficients in the above expressions, we obtain
\begin{align}
\w(x+y)&=\w(x)+\w(y)+\sum_{(\maltese)}\frac{1}{\sharp\{x\}}m[x_1,\cdots,x_p] ;\\
\nonumber
\w_1(x+y)&=\w_1(x)+\w_1(y)+\sum_{(\maltese)}\left( \sum_{k=0}^{p-2}m\left[m_1(m[x_1,\cdots,x_{p-k-1}],x_{p-k}),x_{p-k+1},\cdots,x_p\right]\right)\\
&=\w_1(x)+\w_1(y)+\sum_{(\maltese)}\left( \sum_{k=0}^{p-2}(-1)^{k}m\left(x_p x_{p-1}\cdots x_{p-k-1}m_1([x_1,\cdots,x_{p-k-1}],x_{p-k})\right)\right).\end{align}
\end{enumerate}

We conclude that $\w_1$ has the $(*)$-property with respect to $m_1$.
	
\end{proof}
\noindent\textbf{Remark on the adjoint action.} we recall that $m=[\cdot,\cdot]~$; and if $x,y\in L$, we have $\ad_y(x)=[y,x]=-[x,y]$, so $y\cdot x=\ad_y(x)=-m(x,y)$.

\bigskip

\noindent\textbf{Example: the case $p=3$.} Let $(m,\w)$ be a restricted multiderivation associated to a restricted Lie-Rinehart algebra $\left(A,L, [\cdot,\cdot],(\cdot)^{[p]},\rho\right)$. Consider an infinitesimal deformation $(m_t,\w_t)$ of $(m,\w)$ given by $m_t=m+tm_1,~\w_t=\w+t\w_1$, and let $x,y\in L$.

\begin{align*}
\w_t(x+y)-\w_t(x)-\w_t(y)=&\sum_{\underset{x_{1}=x,~x_2=y}{x_k\in\{x,y\}}}\frac{1}{\sharp\{x\}}m_t(m_t(x_1,x_2),x_3)\\
=&~2m_t(m_t(x,y),x)+m_t(m_t(x,y),y)\\
=&~2\left([[x,y],x]+tm_1([x,y],x)+t[m_1(x,y),x] \right)\\&~+[[x,y],y]+tm_1([x,y],y)+t[m_1(x,y),y].\\
\end{align*}
Collecting the coefficients of $t$ on both sides, we obtain
$$\w_1(x+y)-\w_1(x)-\w_2(x)=2m_1([x,y],x)+2m_1[m_1(x,y),x]+m([x,y],y)+[m(x,y),y],$$ which is exactly saying that $\w_1$ satisfies the $(*)$-property with respect to $m_1$.

\begin{thm}\label{restdefo}
		Let $\left(m_t, \w_t \right)$ be a restricted deformation of $\left(m,\w \right)$. Then $(m_1,\w_1)$ is a 2-cocyle of the restricted cohomology.
\end{thm}

\begin{proof}
	By the previous Lemma, $(m_1,\w_1)\in C^2_*(L,L)$.
	 By the ordinary theory of deformations, we already have that $m_1\in Z^2_{CE}(L,L)$. It remains to show that the induced map $\ind^2(m_1,\w_1)$ vanishes. We expand the equation 
	 \begin{equation}
	 m\left(x,\w_t(y) \right)_t=m_t( m_t(\cdots m_t(x,  \overset{p\text{ terms}}{\overbrace{y),y),\cdots,y}}).
	 \end{equation}
On one hand, we have 

\begin{align*}
m_t\left(x,\w_t(y) \right)&=m\left(x,\w_t(y) \right)+t m_1(x,\w_t(y))+\sum_{i\geq 2}t^im_i(x,\w_t(y))\\
&=m\left(x,\w(y) \right)+t m(x,\w_1(y))+t m_1(x,\w(y)) \mod(t^2).
\end{align*}

On the other hand, we have
\begin{align*}
m_t( m_t(\cdots m_t(x,  \overset{p\text{ terms}}{\overbrace{y),y),\cdots,y}})&
=\displaystyle \sum_{i_p}\cdots\sum_{i_1}t^{i_1+\cdots+i_p}m_{i_p}\left(m_{i(p-1)}(\cdots(m_{i_1}(x,y),y),\cdots,y),y \right)\\
&=m[x,y, \cdots ,y]+t\sum_{\underset{\sharp\left\lbrace k,~i_k=1 \right\rbrace=1 }{i_k=0 \text{ or } 1}}m_{i_p}\left(\cdots(m_{i_1}(x,y),y),\cdots,y),y \right) \mod(t^2)\\
&=m[x,y, \cdots ,y]+t\sum_{i+j=p-1}m\left[m_1(m[x,\overset{j}{\overbrace{y,\cdots,y}}],y),\overset{i}{\overbrace{y,\cdots,y}}   \right] \mod(t^2)\\
&=m[x,y, \cdots ,y]+t\sum_{i+j=p-1}(-1)^iy^im_1(m[x,\overset{j}{\overbrace{y,\cdots,y}}],y) \mod(t^2).\\
\end{align*}
We finally obtain the equation
\begin{equation}
m\left(x,\w(y) \right)+t\left( m(x,\w_1(y))+m_1(x,\w(y))\right) =m[x,y, \cdots ,y]+t\sum_{i+j=p-1}(-1)^iy^im_1(m[x,\overset{j}{\overbrace{y,\cdots,y}}],y)
\end{equation}
By collecting the coefficients of $t^0$ and $t$, we recover the usual identity $m(x,\w(y))=m[x,y,\cdots,y]$ and obtain the new identity
\begin{equation} \left( m(x,\w_1(y))+m_1(x,\w(y))\right)=\sum_{i+j=p-1}(-1)^iy^im_1([x,\overset{j}{\overbrace{y,\cdots,y}}],y),  
\end{equation}
which is equivalent to $\ind^2(m_1,\w_1)=0$. We conclude that $(m_1,\w_1)$ is a 2-cocycle of the restricted cohomology.
\end{proof}

\begin{rmq}
    By forgetting the Lie-Rinehart structure and constructing a deformation $\left(L[[t]],m_t,\w_t \right)$ of the restricted Lie algebra $\left(L[[t]],m,\w \right)$, we recover the same result for restricted Lie algebras, initially proven in \cite{EF08}.
\end{rmq}
\subsection{Equivalence of restricted formal deformations}

Let $\phi:L[[t]]\longrightarrow L[[t]]$ be a formal automorphism defined on $L$ by 
$$ \phi(x)=\sum_{i\geq 0}t^i\phi_i(x),~\phi_i:L\longrightarrow L~\F\text{-linear },~\phi_0=\text{id},  $$ and then extended by $\F[[t]]$-linearity.

\begin{defi}
	Let $\left(m_t, \w_t \right)$ and $\left(m'_t, \w'_t \right)$ be two formal deformations of $\left(m,\w\right)$. They are said to be \textbf{equivalent} (by $\phi$) if for all $x,y\in L$, 
		\begin{equation}
		m_t(\phi(x),\phi(y))=\phi\left(m'_t(x,y) \right)
		\end{equation}
		and
		\begin{equation}	
		\phi\left(\w_t(x) \right)=\w_t'\left(\phi(x)\right). 
\end{equation}
\end{defi}

\begin{lem}\label{equivalem}
	Let $\left(m_t, \w_t \right)$ and $\left(m_t', \w'_t \right)$ be two equivalent formal deformations of $\left(m,\w \right)$. Then it exists $\psi: L\longrightarrow L$ such that, for all $ x,y\in L$,
	\begin{equation}
		 m_1(x,y)-m_1'(x,y)=\psi\left(m(x,y)\right)-m\left(x,\psi(y)\right)-m\left(\psi(x),y \right)
	\end{equation}
	and
	\begin{equation}
		 \w_1(x)-\w_1'(x)=m[\psi(x), \overset{p-1}{\overbrace{x,\cdots,x}}]-\psi(\w(x)).
		 \end{equation}  
 If the equivalence is given by $\phi=\displaystyle \sum_{i\geq 0} t^i \phi_i$, then $\psi=\phi_1$.
\end{lem}

\begin{proof}
	$$m_t(\phi(x),\phi(y))=\phi\left(m_t'(x,y) \right)\Longleftrightarrow \sum_{k\geq 0}t^km_t(\phi(x),\phi(y))=\sum_{k\geq 0}t^k\phi\left(m_t'(x,y) \right).     $$
	We deduce that
	\begin{align*}
	&~~~~~m(\phi(x),\phi(y))+tm_1(\phi(x),\phi(y))
	=\phi(m(x,y))+t\phi(m_1'(x,y))\mod(t^2)\\
	&\Rightarrow m\left(\sum_{i\geq 0}t^i\phi_i(y),\sum_{j\geq 0}t^j\phi_j(y) \right)+tm_1\left(\sum_{i\geq 0}t^i\phi_i(y),\sum_{j\geq 0}t^j\phi_j(y) \right)\\~~&=\sum_{i\geq 0}t^i\phi_i(m(x,y))+t\sum_{j\geq 0}t^j\phi_j(m_1'(x,y))\mod(t^2)\\
	&\Rightarrow m(x,y)+t\left( m(x,\phi_1(y))+m(\phi_1(x),y)+m_1(x,y) \right)= m(x,y)+t\left(\phi_1(m(x,y))+m_1'(x,y) \right)\mod(t^2). 
\end{align*}
	By collecting the coefficients of $t$, we obtain the first identity. Then,
	\begin{align*}
	\phi\left(\w_t(x) \right)=\w\left(\phi(x)\right)
	&\Rightarrow \phi\left(\sum_{i>0}t^i\w_i(x) \right)=\w_t'\left(x+t\phi_1(x) \right)\mod (t^2)\\  
	&\Rightarrow \w(x)+t\left(\w_1(x)+\phi_1(\w(x)) \right)=\w_t'\left(x+t\phi_1(x) \right) \mod(t^2).
	\end{align*}
	
Now we compute the right hand side  of the above equation. We denote once again by $(\maltese)$ the conditions $( x_i\in\{x,y\},~x_1=x,~x_2=y )$, where we temporarily denote $y=t\phi_1(x)$. We then have
\begin{align*}
\w_t'\left(x+t\phi_1(x) \right)&=\w_t'(x)+\w_t'(y)+\sum_{(\maltese)}\frac{1}{\sharp\{x\}}m_t[x_1,\cdots,x_p]\\
&=\w_t'(x)+\w_t'(y)+\sum_{(\maltese)}\frac{1}{\sharp\{x\}}m[x_1,\cdots,x_p]\mod(t^2)\\
&=\w_t'(x)+\w_t'(t\phi_1(x))+\frac{1}{p-1}m[x,t\phi_1(x),x,x,\cdots,x]\mod(t^2)\\
&=\w_t'(x)-tm[x,\phi_1(x),x,x,\cdots,x]\mod(t^2)\\
&=\w(x)+t\left(\w_1'(x)-m[x,\phi_1(x),x,x,\cdots,x] \right) \mod(t^2).
\end{align*}
Using this result, we obtain
$$\w(x)+t\left(\w_1(x)+\phi_1(\w(x)) \right)=\w(x)+t\left(\w_1'(x)-m[x,\phi_1(x),x,x,\cdots,x] \right) \mod(t^2).       $$
Collecting the coefficients of $t$ in the previous equation, we obtain
$$ \w_1(x)-\w_1'(x)=m[\phi_1(x),\overset{p-1}{\overbrace{x,\cdots,x}}]-\phi_1(\w(x)),$$ which is the identity we wanted.
\end{proof}

\begin{rmq}
	Lemma \ref{equivalem} justifies the definitions given in \cite{EF08} and \cite{ET00} in the case of infinitesimal deformations of restricted Lie algebras.
\end{rmq}

\begin{thm}
		Let $\left(m_t, \w_t \right)$ and $\left(m_t', \w'_t \right)$ be two formal deformations of $\left(m,\w\right)$. Then, their infinitesimal elements are in the same cohomology class.
\end{thm}
\begin{proof}
		Let $\left(m_t, \w_t \right)$ and $\left(m_t', \w'_t \right)$ be two equivalent formal deformations via $\phi=\displaystyle\sum_{i\geq 0}t^i\phi_i$. By Lemma \ref{equivalem}, we have
	\begin{equation}\label{CEeq}
	 m_1(x,y)-m_1'(x,y)=\phi_1\left(m(x,y)\right)-m\left(x,\phi_1(y)\right)-m\left(\phi_1(x),y \right)\text{ and }
	 \end{equation}
	 \begin{equation}\label{pmapeq}
	\w_1(x)-\w_1'=m[\phi_1(x), \overset{p-1}{\overbrace{x,\cdots,x}}]-\psi(\w(x)).     
	\end{equation}

We can easily check that $\phi_1$ belongs to $C^1_{CE}(L,L)$. We are then able to compute $d^1_*(\phi_1)=\left( d^1_{CE}(\phi_1),\ind^1(\phi_1)\right) $.
\begin{align*}
\left(d^1_{CE}(\phi_1) \right)(x,y)&=x\cdot\phi_1(y)-y\cdot\phi_1(x)-\phi_1(m(x,y))\\ 
&=m(x,\phi_1(y))-m(y,\phi_1(x))-\phi_1(m(x,y))\\
&=m(x,\phi_1(y))+m(\phi_1(x),y)-\phi_1(m(x,y))\\
&=-\left(\phi_1(m(x,y))-m(x,\phi_1(y))-m(\phi_1(x),y) \right).
\end{align*}
Using Equation (\ref{CEeq}), we deduce that $m_1(x,y)-m_1'(x,y)=-d^1\phi_1(x,y)$, so $\left( m_1-m_1'\right) \in B^2_{CE}(L,L)$. 
\begin{align*}
\ind^1(\phi_1)&=\phi_1\left(\w(x)\right)-m[\phi_1(x),x,x,\cdots,x]\\
&=-\left(\w_1(x)-\w_1'(x) \right).  
\end{align*}	
Finally, $\left(m_1-m_1',~\w_1-\w_1'\right)=-d_*^1\phi_1\in B^2_*(L,L). $ 
We conclude that $(m_1,\w_1)$ and	$(m_1',\w_1')$ differ from a coboundary, so they are in the same cohomology class.
	
\end{proof}

\begin{defi}
    Let $(m_t,\w_t)$ be a restricted deformations of $(m,\w)$.
    The deformation is said to be \textbf{trivial} if there is a formal automorphism $\phi$ such that
       \begin{equation}\label{eqtriv1}
        \phi\left(m(x,y)\right)=m_t(\phi(x),\phi(y))
    \end{equation}
    and
    \begin{equation}\label{eqtriv2}
        \phi\left(\w_t(x)\right)=\w\left(\phi(x)\right).
    \end{equation}
\end{defi}
By expanding the relation (\ref{eqtriv1}) mod $(t^2)$, we obtain that $m_1$ is a Chevalley-Eilenberg coboundary. More precisely, we have $m_1=-d_{CE}^1(\phi_1)$. Now, we focus on the right-hand side of Equation (\ref{eqtriv2}). The following computations are made mod $t^2$, for all  $x\in L$:
\begin{align*}
    \w\left(\phi(x)\right)&= \w\left( \sum_it^i\phi_i(x) \right)\\
                    &=\w\left(x+t\phi_1(x)\right)\\
                    &=\w(x)+\w\left(t\phi_1(x)\right)+\displaystyle\sum_{\underset{x_{p-1}=t\phi_1(x),~x_p=x}{x_i\in\{x,~t\phi_1(x)\}}}\frac{1}{\sharp\{x\}} m(x_1,m(\cdots,m(x_{p-1},x_p)\cdots).   \\
                    &=\w(x)+\w\left(t\phi_1(x)\right)+\frac{1}{p-1}m(x,m(x,\cdots m(t\phi_1(x),x)\cdots)\\
                    &=\w(x)+\w\left(t\phi_1(x)\right)+t\ad_x^{p-1}\circ\phi_1(x).
\end{align*}
For the left-hand side, we have (mod $t^2$):
\begin{align*}
    \phi\left(\w_t(x)\right)&=\left(\sum_it^i\phi_i\left(\w_t(x)\right)\right)
    =\sum_i\sum_jt^{i+j}\phi_i\left(\omega_j(x)\right)
    =\w(x)+t\left( \omega_1(x)+\phi_1\left(\w(x)\right) \right).
    \end{align*}
We deduce that 
\begin{equation}
    \omega_1(x)+\phi_1\left(\w(x)\right)=\ad_x^{p-1}\circ\phi_1(x),
\end{equation}
which can be rewritten
\begin{equation}
    \omega_1(x)=-\phi_1\left(\w(x)\right)+\ad_x^{p-1}\circ\phi_1(x)=-\ind^1(\phi_1)(x).
\end{equation}
Thus, we  have proven the following proposition.

\begin{prop}
    The deformation $\left( m+tm_1, \w+t\w_1 \right)$ is trivial  if and only if  $(m_1,\w_1)$  is a restricted coboundary.
\end{prop}

\begin{rmq}
    By forgetting the Lie-Rinehart structure and constructing a deformation $\left(L[[t]],m_t,\w_t \right)$ of the restricted Lie algebra $\left(L[[t]],m,\w \right)$, we recover the same results for restricted Lie algebras, initially proven in \cite{EF08}. Moreover, we introduced the notion of restricted trivial deformation.
\end{rmq}
\subsection{Obstructions}
Let $\left(A,L, [\cdot,\cdot],(-)^{[p],\rho} \right)$ be a restricted Lie-Rinehart algebra, $(m,\w)$ be the associated restricted multiderivation and $n\geq 1$. A deformation is of order $n$ if it is given by
$$m^n_t=\sum_{i=0}^{n}t^im_i;~~~\w_t^n=\sum_{i=0}^{n}t^i\w_i.$$

\begin{defi}
	Let $\left( m_t^n,\w_t^n \right) $ be an $n$-order deformation of $\left(m,\w\right)$. We set for all  $x,y,z\in L,$ 	
	\begin{align*}
	\obs^{(1)}_{n+1}(x,y,z)&=\sum_{i=1}^{n}\left(m_i(x,m_{n+1-i}(y,z))+m_i(y,m_{n+1-i}(z,x))+m_i(z,m_{n+1-i}(x,y)) \right);\\
	\obs^{(2)}_{n+1}(x,y)~~&=\sum_{\underset{i_1+\cdots+i_p=n+1}{0\leq i_k\leq n}}m_{i_p}\left(m_{i_{p-1}}(\cdots(m_{i_1}(x,y),y),\cdots,y),y \right)-\sum_{i=1}^{n}m_i(x,\w_{n+1-i}(y)).\\
	\end{align*}

\end{defi}

\begin{prop}
		Let $\left( m_t^n,\w_t^n \right) $ be a $n$-order deformation of $\left(m,\w\right)$. Let $\left(m_{n+1},~\w_{n+1} \right)\in C^2_*(L,L)$. Then $\left( m^n_t+t^{n+1}m_{n+1},~  \w^n_t+t^{n+1}\w_{n+1} \right) $ is a $(n+1)$-order deformation of $(m,\w)$ if and only if  
	
	$$	\left( \obs^{(1)}_{n+1}, \obs^{(2)}_{n+1}\right)=d_*^2\left(m_{n+1},~\w_{n+1} \right).        $$
\end{prop}

\begin{proof}
	Let $\left( m^n_t+t^{n+1}m_{n+1},~  \w^n_t+t^{n+1}\w_{n+1} \right)$ be a $(n+1)$-order deformation of $L$. The deformed element $m_t^n+t^{n+1}m_{n+1}$ satisfies the Jacobi identity, which can be written
	
	$$\sum_{q=0}^{n+1}t^{q}\sum_{i=0}^{q}\left(m_i(x,m_{q-i}(y,z))+m_i(y,m_{q-i}(z,x))+m_i(z,m_{q-i}(x,y)) \right)=0.       $$ By collecting the coefficients of $t^{n+1}$, we obtain
	
	$$ \sum_{i=0}^{n+1}\left(m_i(x,m_{n+1-i}(y,z))+m_i(y,m_{n+1-i}(z,x))+m_i(z,m_{n+1-i}(x,y)) \right)=0.    $$

By isolating the first and last terms of the sum, we have
\begin{align*}
&m(x,m_{n+1}(y,z))+m(y,m_{n+1}(z,x))+m(z,m_{n+1}(x,y)\\ +&m_{n+1}(x,m(y,z))+m_{n+1}(y,m(z,x))+m_{n+1}(z,m(x,y))\\+& \sum_{i=1}^{n}\left(m_i(x,m_{n+1-i}(y,z))+m_i(y,m_{n+1-i}(z,x))+m_i(z,m_{n+1-i}(x,y)) \right)=0. 
\end{align*}  
We recall the expression of the differential of Chevalley-Eilenberg of order two:
$$d^2_{CE}\varphi(x,y,z)=\varphi([y,z],x)-\varphi([x,z],y)+\varphi([x,y],z)-[x,\varphi(y,z)]+[y,\varphi(x,z)]-[z,\varphi(x,y)]. $$

Using the two last equations with $\varphi=m_{n+1}$, recalling that $m=[\cdot,\cdot]$ and by skew-symmetry, we finally obtain
$$\sum_{i=1}^{n}\left(m_i(x,m_{n+1-i}(y,z))+m_i(y,m_{n+1-i}(z,x))+m_i(z,m_{n+1-i}(x,y)) \right)=d^2_{CE}m_{n+1}(x,y,z),$$ which exactly means that $d^2_{CE}m_{n+1}(x,y,z)=\obs^{(1)}_{n+1}(x,y,z)$. 

\vspace{0.5cm}

Denote $\w_t^n+t^{n+1}\w_{n+1}:=\w_t^{n+1}$. If $\w_t^{n+1}$ is a deformation of order $(n+1)$, then it satisfies the equation
\begin{equation}
m_t^{n+1}\left(x,\w_t^{n+1}(y) \right)=m_t^{n+1}( m_t^{n+1}(\cdots m_t^{n+1}(x,  \overset{p\text{ terms}}{\overbrace{y),y),\cdots,y}}).
\end{equation}
By expanding and collecting the coefficients of $t^{n+1}$, we obtain
$$\sum_{i=0}^{n+1}m_i(x,\w_{n+1-i}(y))=\sum_{\underset{i_1+\cdots+i_p=n+1}{0\leq i_k\leq n+1}}m_{i_p}\left(m_{i(p-1)}(\cdots(m_{i_1}(x,y),y),\cdots,y),y \right),    $$
which can be rewritten
\begin{align*}
&m(x,\w_{n+1}(y))+m_{n+1}(x,\w(y))-\sum_{i+j=p-1}m\left[ m_{n+1}(m[x,\overset{j}{\overbrace{y,\cdots,y}}],y),\overset{i}{\overbrace{y,\cdots,y}}\right]\\
=&\sum_{\underset{i_1+\cdots+i_p=n+1}{0\leq i_k\leq n}}m_{i_p}\left(m_{i(p-1)}(\cdots(m_{i_1}(x,y),y),\cdots,y),y \right)-\sum_{i=1}^{n}m_i(x,\w_{n+1-i}(y)).
\end{align*}
Recalling the definition of $\ind^2(m_{n+1},\w_{n+1})$, we finally obtain

$$d^2_*\left(m_{n+1},\w_{n+1} \right) =\left(d^2_{CE}m_{n+1},\ind^2(m_{n+1},\w_{n+1})\right)= \left( \obs^{(1)}_{n+1}, \obs^{(2)}_{n+1}\right).     $$
Conversely, if the latter equation is satisfied, the same computation shows that $$\left( m^n_t+t^{n+1}m_{n+1},~  \w^n_t+t^{n+1}\w_{n+1} \right)$$ is a $(n+1)$-order deformation of $(m,\w)$.
\end{proof}

We shall stop our obstruction computations here. The result which is missing is that $\left( \obs^{(1)}_{n+1}, \obs^{(2)}_{n+1}\right)$ is a $3$-cocycle of the cohomology, but so far we do not have an expression for $d^3_*$, so this result is beyond our reach for the moment.

\begin{rmq}
    By forgetting the Lie-Rinehart structure and constructing a $n$-order deformation $\left(L[[t]],m_t^n,\w_t^n \right)$ of the restricted Lie algebra $\left(L[[t]],m,\w \right)$, we have the same results for restricted Lie algebras, which is also  new to our knowledge.
\end{rmq}
\subsection{Restricted Nijenhuis operators}
In this subsection, we introduce restricted Nijenhuis operator and show their relationships with trivial deformations. \begin{defi}
Let $\left(L,[\cdot,\cdot],(\cdot)^{[p]}\right)$ be a restricted Lie algebra.
    A linear map $N: L\rightarrow L$ is called \textbf{restricted Nijenhuis operator} on $L$ if
    \begin{align}
        N\left([N(x),y]+[x,N(y)]-N([x,y])\right)&=[N(x),N(y)],\\
        N\left(N(x^{[p]})-\ad_x^{p-1}\circ N(x)\right)&=N(x)^{[p]},~~~~~~~~~\forall x,y\in L.
    \end{align}

\end{defi}

Now, we define the following  two maps on $L$: 
\begin{align}
    [x,y]_N&=[N(x),y]+[x,N(y)]-N([x,y]),\\
    x^{[p]_N}&=N(x^{[p]})-\ad_x^{p-1}\circ N(x).
\end{align}

\begin{prop}
    The pair $\left([\cdot,\cdot]_N,(\cdot)^{[p]}\right)$ is a restricted $2$-cocycle. Moreover, the restricted formal deformation given by
\begin{equation}
   [x,y]_t=[x,y]+t[x,y]_N,~~x^{[p]_t}=x^{[p]}+t x^{[p]_N}
\end{equation}
 is trivial.
\end{prop}
\begin{proof}
    By definition, we have
    \begin{align*}
    [x,y]_N&=d^1_{CE}N(x,y),\\
    x^{[p]_N}&=\ind^1N(x).
\end{align*}
Therefore, $\left([\cdot,\cdot]_N,(\cdot)^{[p]}\right)$ is a restricted $2$-coboundary.
\end{proof}

\section{Restricted Heisenberg algebras}\label{sectionh}

The origins of quantum mechanics lie in the groundbreaking idea of Heisenberg to consider the components of the position vector $\textbf{q}=(q_1,q_2,q_3)\in \R^3$ and the momentum vector $\textbf{p}=(p_1,p_2,p_3)\in\R^3$ of a particle at a given time $t$ as operators on a Hilbert space. Those components have to satisfy the relations
$$[q_j,q_k]=0,~[p_j,p_k]=0,~[p_j,q_k]=-i\hbar\delta_{j,k},$$
$i$ being the square root of $-1$, $j,k\in \{1,2,3\}$ and $\hbar$ the Planck constant. Those relations can be extended to any vectors $\textbf{q},\textbf{p}\in \R^n$. Usually, one considers those commutation relations as defining relations of a Lie algebra of dimension $2n+1$ by adding a central element $z$ such that
$$[q_j,q_k]=[p_j,p_k]=[p_j,z]=[q_j,z]=0,~[p_j,q_k]=\delta_{j,k}z.$$ This $(2n+1)$-dimensional real Lie algebra is called the Heisenberg algebra. It can be seen as a central extension of the commutative algebra $\R^{2n}$ by a copy of $\R$ spanned by $z$. More background material can be found in \cite{WP17}.

In this section, we investigate examples based on the Heisenberg Lie algebra of dimension $3$ over fields of characteristic $p\geq 3$. We study the restricted structure of the Heisenberg algebra, then we give an explicit description of the second restricted cohomology spaces. Finally, we give an example of restricted Lie-Rinehart structure involving the restricted Heisenberg algebras and discuss its deformations.
Restricted $p$-nilpotent Heisenberg algebras have also been considered in \cite{SU16}.
\subsection{Restricted structures on the Heisenberg algebra}

Let $\mathbb{F}$ be a field of characteristic $p\geq3$. We consider the \textbf{Heisenberg algebra} $\h=\text{Span}_{\F} \{x,y,z\}$ defined with the bracket $[x,y]=z.$ This Lie algebra is nilpotent of order $2$, therefore all the $p$-folds brackets on $\h$ vanish. Let $(\cdot)^{[p]}$ be a $p$-map on $\h$. We then have $(u+v)^{[p]}=u^{[p]}+v^{[p]},$ for all $u,v\in \h$. Hence, any $p$-map on $\h$ is $p$-semilinear.
 
 \begin{rmq} 
    In \cite{EM22}, $\h$ is denoted $L^2_{3|0}$.
\end{rmq}

 \begin{prop}\label{classifheisenberg}
Any $p$-structure on $\h$ is given by $x^{[p]}=\theta(x)z,~y^{[p]}=\theta(y)z,~z^{[p]}=\theta(z)z,$
with  $\theta: \h\rightarrow \F$ a linear form on $\h$.
 \end{prop}
 
 \begin{proof}
Using Theorem \ref{jacobson}, it is enough to check that  
$$\left(\ad_{x} \right)^p-\theta(x)\ad_{z}=\left(\ad_{y} \right)^p-\theta(y)\ad_{z}=\left(\ad_{z} \right)^p-\theta(z)\ad_{z}=0$$ to obtain the first claim. The above identities are always true, because $z$ lies in the center of $\h$. Conversely, let $(\cdot)^{[p]}$ be a $p$-map on $\h$. Because of the second condition of the Definition \ref{restdefi}, the image of $(\cdot)^{[p]}$ lies in the center of $\h$, which is one-dimensional and spanned by $z$. 
Therefore, it exists $\theta:\h\rightarrow \F$ linear such that $x^{[p]}=\theta(x)z,~y^{[p]}=\theta(y)z,~z^{[p]}=\theta(z)z.$
 \end{proof}
 
\noindent\textbf{Notation.} We will denote a restricted Heisenberg algebra whose $p$-map is given by the linear form $\theta$ by $(\h, \theta)$. 

\begin{rmq}
Let $u\in(\h,\theta),~u=\alpha x+\beta y + \gamma z,~\alpha,\beta,\gamma\in\F.$ Then $u^{[p]}=\left(\alpha^p\theta(x)+\beta^p\theta(y)+\gamma^p\theta(z)\right)z.$
\end{rmq}

\begin{lem}\label{heisiso}
Let $(\h,\theta)$ and $(\h,\theta')$ be two  restricted Heisenberg algebras. Then, any Lie isomorphism $\phi:(\h,\theta)\rightarrow (\h,\theta')$ is of the form 
\begin{equation}\begin{cases}\label{hiso}
        \phi(x)&=ax+by+cz\\
        \phi(y)&=dx+ey+fz\\
        \phi(z)&=(ae-bd)z,~~ae-bd\neq0,
\end{cases}\end{equation}
with $a,b,c,d,e,f\in\F$. Moreover, $\phi$ is a restricted Lie isomorphism if and only if
 \begin{equation}\begin{cases}\label{resthiso}
        \theta(x)u&=a^p\theta'(x)+b^p\theta'(y)+c^p\theta'(z)\\
        \theta(y)u&=d^p\theta'(x)+e^p\theta'(y)+f^p\theta'(z)\\
        \theta(z)u&=u^p\theta'(z),
\end{cases}\end{equation}
    where $u:=ae-bd\neq0$.
\end{lem}

\begin{proof}
    A Lie isomorphism $\phi:\h\rightarrow\h$ must satisfy $\phi([v,w])=[\phi(v),\phi(w)],\text{ with }v,w\in\h,$ as well as $\det(\phi)=0$. Applying those conditions to any linear map $\phi:\h\rightarrow\h$, we obtain Conditions (\ref{hiso}). Then, $\phi$ is a restricted map on $\h$ if and only if $\phi\left(v^{[p]}\right)=\phi(v)^{[p]'}$, with $v\in \h$ and $(\cdot)^{[p]'}$ the $p$-map on $\h$ given by the linear form $\theta'$. We obtain Conditions (\ref{resthiso}) by evaluating this equation on the basis of $\h$:
    \begin{align*}
       \phi\left(x^{[p]}\right)=\phi(x)^{[p]'}&\implies\theta(x)\phi(z)=(ax+by+cz)^{[p]'}\\
       &\implies\theta(x)uz=\theta'(x)z+b^p\theta'(y)z+c^p\theta'(z)z\\
       &\implies\theta(x)u=\theta'(x)+b^p\theta'(y)z+c^p\theta'(z).
    \end{align*}
   Other two equations can be  obtained in a similar way.
\end{proof}

\begin{thm}\label{heisclass}
    There are three non-isomorphic restricted Heisenberg algebras, respectively given by the linear forms $\theta=0,$ $\theta=x^*$ and $\theta=z^*$.
\end{thm}

\begin{proof}
    \begin{itemize}
    \item[$\bullet$] First, we will show that $(\h,x^*)$ is isomorphic to $(\h,y^*)$. By setting $\theta=x^*$ and $\theta'=y^*,$ Conditions (\ref{resthiso}) reduce to $\{u=b^p,~e^p=0\}$. It suffices to choose $e=0, b\neq 0$ and $d=-b^{p-1}$ to build a suitable restricted isomorphism between $(\h,x^*)$ and $(\h,y^*)$.
    \item[$\bullet$] Let $\theta=0$ and $\theta'=x^*$. Then, Conditions (\ref{resthiso}) reduce to $\{a^p=0,~d^p=0\}$. But, this is impossible since $u=ae-bd\neq0$. Therefore, $(\h,0)$ and $(\h,x^*)$ are not isomorphic.
    \item[$\bullet$] Let $\theta=0$ and $\theta'=z^*$. Then, Conditions (\ref{resthiso}) reduce to $\{c^p=0,~f^p=0,~u^p=0\}$. But, this is impossible since $u\neq0$. Therefore, $(\h,0)$ and $(\h,z^*)$ are not isomorphic.
    \item[$\bullet$] Let $\theta=x^*$ and $\theta'=z^*$. Then, Conditions (\ref{resthiso}) reduce to $\{c^p=u,~f^p=0,~u^p=0\}$. But, this is impossible since $u\neq0$. Therefore, $(\h,x^*)$ and $(\h,z^*)$ are not isomorphic.
    \end{itemize}
\end{proof}

\begin{rmq} The restricted algebras $(\h,0)$ and $(\h,x^*)$ appeared in \cite{SU16} and are $p$-nilpotent.
\end{rmq}

\subsection{Restricted cohomology of restricted Heisenberg algebras}

In this section, we compute the second restricted cohomology groups of the restricted Heisenberg algebras with adjoint coefficients. Let $\theta$ be a linear form on the (ordinary) Heisenberg algebra. We denote by $(\h,\theta)$ the restricted Heisenberg algebra obtained with $\theta$ (see Proposition \ref{classifheisenberg}). We also denote $H^2_*(\h,\theta)=H^2_*((\h,\theta),(\h,\theta))$ the second restricted cohomology group of $(\h,\theta)$ with adjoint coefficients.
\bigskip

\subsubsection{The case $p>3$}

Let $\F$ be a field of characteristic $p>3$ and let $\varphi\in C_{CE}^2(\h,\h)$.
Since  the (ordinary) Heisenberg algebra $\h$ is nilpotent of order $2$ and $p>3$, any $p$-semilinear map $\w:\h\rightarrow\h$ satisfies the $(*)$-property with respect to $\varphi$. Therefore, for any linear form $\theta$ on $\h$, we have $C_{*}^2(\h,\theta)=\Hom_{\F}(\Lambda^2\h,\h)\oplus\Hom_{\F}(\overline{\h},\h)$ as vector spaces.

\begin{lem}\label{ordcocycle}
Let $\h$ be the ordinary Heisenberg algebra. Let $\varphi\in C_{CE}^2(\h,\h)$ given by
\begin{equation}\label{stdrphi}
    \begin{cases}
            \varphi(x,y)&=ax+by+cz\\
            \varphi(x,z)&=dx+ey+fz\\
            \varphi(y,z)&=gx+hy+iz
    \end{cases}
\end{equation}
with parameters $a,b,c,d,e,f,g,h,$ belonging to $\F$. Then, $\varphi$ is a $2$-cocycle of the Chevalley-Eilenberg cohomology if and only if $h=-d$.
\end{lem}

\begin{proof}
The only non-trivial 2-cocycle condition on the basis $\{x,y,z\}$ of $\h$ is
\begin{equation}
\varphi([x,y],z)-\varphi([x,z],y)+\varphi([y,z],x)=[x,\varphi(y,z)]-[y,\varphi(x,z)]+[z,\varphi(x,y)],
\end{equation}
which reduces to $(h+d)z=0.$
\end{proof}
Let $(\varphi,\w)\in C^2_*(\h,\h)$. As the $k$-folds brackets vanish for $k>2$ and $p>3$, we have

\begin{equation}\label{pcocyl}
    \ind^2(\varphi,\w)(v,w)=\varphi\left(v,w^{[p]}\right)+[v,\w(w)],~\forall v,w\in \h.
\end{equation}

\begin{lem}\label{hrescocy}
The restricted $2$-cocycles for $(\h,\theta)$ are given by pairs $(\varphi,\w)$, where
\begin{itemize}
    \item[$\bullet$] \underline{Case $\theta=0$}:
  
        \begin{equation}\begin{cases}
            \varphi(x,y)&=ax+by+cz\\
            \varphi(x,z)&=dx+ey+fz\\
            \varphi(y,z)&=gx-dy+iz\\
        \end{cases}
        ~~~~~~~~~~
        \begin{cases}
            \w(x)&=\gamma z\\
            \w(y)&=\epsilon z\\
            \w(z)&=\kappa z\\
        \end{cases}\end{equation}

    \item[$\bullet$] \underline{Case $\theta=x^*$}:
  
        \begin{equation}\begin{cases}
            \varphi(x,y)&=ax+by+cz\\
            \varphi(x,z)&=fz\\
            \varphi(y,z)&=iz\\
        \end{cases}
        ~~~~~~~~~~
        \begin{cases}
            \w(x)&=ix-fy+\gamma z\\
            \w(y)&=\epsilon z\\
            \w(z)&=\kappa z;\\
        \end{cases}\end{equation}

    \item[$\bullet$] \underline{Case $\theta=z^*$}:
  
        \begin{equation}\begin{cases}
            \varphi(x,y)&=ax+by+cz\\
            \varphi(x,z)&=fz\\
            \varphi(y,z)&=iz;\\
        \end{cases}
        ~~~~~~~~~~
        \begin{cases}
            \w(x)&=\gamma z\\
            \w(y)&=\epsilon z\\
            \w(z)&=ix-fy+\kappa z,\\
        \end{cases}\end{equation} where all the parameters $a,b,c,d,e,f,h,i,\gamma,\epsilon,\kappa$ belong to $\F$.
\end{itemize}
\end{lem}

\begin{proof}
Let $\varphi\in Z_{CE}^2(\h,\h)$ given by Lemma \ref{ordcocycle}, let $\theta$ be a linear form on $\h$ and let $\w:\h\rightarrow\F$ a map having the $(*)$-property w.r.t $\varphi$, given on the basis of $\h$ by
\begin{equation}
        \begin{cases}
            \w(x)&=\alpha x+\beta y+\gamma z\\
            \w(y)&=\lambda x+\mu y+\epsilon z\\
            \w(z)&=\delta x+\eta y+\kappa z,\\
        \end{cases}
\end{equation}
with $\alpha,\beta,\gamma,\lambda,\mu,\epsilon,\delta,\eta,\kappa$ belonging to $\F$. Suppose moreover that $(\varphi,\w)\in Z_*^2(\h,\theta)$.

\begin{itemize}
\item The case $\theta=0$. We evaluate  Equation (\ref{pcocyl}) on elements of the basis $\{x,y,z\}$ of $(\h,0)$.
\begin{align*}
0&=\varphi\left(x,x^{[p]}\right)+[x,\w(x)]=\varphi\left(x,0\right)+[x,\beta y]+[x,\gamma z]=\beta z\implies\beta=0\\
0&=\varphi\left(y,y^{[p]}\right)+[y,\w(y)]=\varphi\left(y,0\right)+[y,\lambda x]+[y,\epsilon z]=\lambda z\implies\lambda=0;\\
0&=\varphi\left(x,y^{[p]}\right)+[x,\w(y)]=\varphi\left(x,0\right)+[x,\mu y]+[x,\epsilon z]=\mu z\implies\mu=0;\\
0&=\varphi\left(x,z^{[p]}\right)+[x,\w(z)]=\varphi\left(x,0\right)+[x,\eta z]+[x,\kappa z]=\eta z\implies\eta=0;\\
0&=\varphi\left(y,z^{[p]}\right)+[y,\w(z)]=\varphi\left(y,0\right)+[y,\delta x]+[y,\kappa z]=-\delta z\implies\delta=0;\\
0&=\varphi\left(y,x^{[p]}\right)+[y,\w(x)]=\varphi\left(y,0\right)+[y,\alpha x]+[y,\gamma z]=-\alpha z\implies\alpha=0.
\end{align*}The other possible equations obtained with $(z,z),~(z,y)$ and $(z,x)$ are trivial. 

\item The case $\theta=x^*$. We evaluate Equation (\ref{pcocyl}) on elements of the basis $\{x,y,z\}$ of $(\h,x^*)$.
\begin{align*}
0&=\varphi\left(x,x^{[p]}\right)+[x,\w(x)]=\varphi\left(x,z\right)+[x,\beta y]=dx+ey+fz+\beta z\implies\beta=-f,~d=e=0;\\
0&=\varphi\left(y,y^{[p]}\right)+[y,\w(y)]=\varphi\left(y,0\right)+\lambda[y, x]=-\lambda z \implies\lambda=0;\\
0&=\varphi\left(x,y^{[p]}\right)+[x,\w(y)]=\varphi\left(x,0\right)+\mu[x,y]=\eta z \implies\mu=0;\\
0&=\varphi\left(x,z^{[p]}\right)+[x,\w(z)]=\varphi\left(x,0\right)+\eta[x,y]=\eta z \implies\eta=0;\\
0&=\varphi\left(y,z^{[p]}\right)+[y,\w(z)]=\varphi\left(y,0\right)+\delta[y,x]=-\delta z \implies\delta=0;\\
0&=\varphi\left(y,x^{[p]}\right)+[y,\w(x)]=\varphi\left(y,z\right)+\alpha[y,x]=gx+iz-\alpha z\implies \alpha=i,~g=0.
\end{align*}The other possible equations obtained with $(z,z),~(z,y)$ and $(z,x)$ are trivial. 

\item The case $\theta=z^*$ is analog to the case $\theta=x^*$.
\end{itemize}
\end{proof}

\begin{lem}\label{hrestcob}
The restricted $2$-coboundaries for $(\h,\theta)$ are given by pairs $(\varphi,\w)$, where

        $$\begin{cases}
            \varphi(x,y)&=Ax+By+\tilde{C}z\\
            \varphi(x,z)&=-Hz\\
            \varphi(y,z)&=Gz,\\
        \end{cases}$$ with $A,B,\tilde{C},G,H$ belonging to $\F$ and
\begin{itemize}
    \item[$\bullet$] \underline{Case $\theta=0$}: $\w=0;$ 

   \item[$\bullet$] \underline{Case $\theta=x^*$}: $\w(x)=Gx+Hy+Iz,~\w(y)=\w(z)=0;$
  
    \item[$\bullet$] \underline{Case $\theta=z^*$}:
    $\w(x)=\w(y)=0,~\w(z)=Gx+Hy+Iz$.
\end{itemize}

\end{lem}

\begin{proof}
Let $\varphi\in C^2_{CE}(\h,\h),$ given on the basis of $\h$ by Equation (\ref{stdrphi}). Suppose that $\varphi=d^1_{CE}\psi,$ with $\psi:\h\rightarrow\h$ given by
\begin{equation}
\begin{cases}
            \psi(x,y)&=Ax+By+Cz\\
            \psi(x,z)&=Dx+Ey+Fz\\
            \psi(y,z)&=Gx+Hy+Iz,
        \end{cases}
\end{equation}
with $A,B,C,D,E,F,G,H,I\in \F$. Using the coboundary condition $\varphi=d^1_{CE}\psi,$ it is easy to show that $d=e=g=h=0,~a=i=G,~b=-f=H,~c=\Tilde{C},$ with $\tilde{C}=I-E-A$. For the restricted part, suppose that $(\varphi,\w)\in B^2_*(\h,\theta)$. The coboundary condition is then given by
\begin{equation}\label{cob}
    \w(u)=\psi\left(u^{[p]}\right)-\ad_u^{p-1}\circ~\psi(u),~u\in \h.
\end{equation}

By evaluating Equation (\ref{cob}) on the basis of $\h$, we obtain

\begin{equation}
\begin{cases}
            \w(x)&=\theta(x)(Gx+Hy+Iz)\\
            \w(y)&=\theta(y)(Gx+Hy+Iz)\\
            \w(z)&=\theta(z)(Gx+Hy+Iz).
        \end{cases}
\end{equation}
Choosing $\theta$ in $\{0,x^*,z^*\}$, we obtain the result.

\end{proof}

\begin{thm}
We have $\dim_{\F}\left(H^2_*(\h,0)\right)=8$ and $\dim_{\F}\left(H^2_*(\h,x^*)\right)=\dim_{\F}\left(H^2_*(\h,z^*)\right)=4.$
\begin{itemize}
\item[$\bullet$] A basis for $H^2_*(\h,0)$ is given by $\{(\varphi_1,
0), (\varphi_2,0), (\varphi_3,0), (\varphi_4,0), (\varphi_5,0), (0,\w_1), (0,\w_2), (0,\w_3)\}$, with
\begin{align*}
\varphi_1(x,z)&=z;~\varphi_2(y,z)=z;~\varphi_3(x,z)=-\varphi_3(y,z)=x;~\varphi_4(x,z)=y;~\varphi_5(y,z)=y;\\
\w_1(x)&=z;~\w_2(y)=z;~\w_3(z)=z.
\end{align*}
(We only write non-zero images).
\item[$\bullet$] A basis for $H^2_*(\h,x^*)$ is given by $\{(\varphi_1,
0), (\varphi_2,0), (0,\w_1), (0,\w_2)\}$, with
$$\varphi_1(x,y)=x;~\varphi_2(x,y)=y;~
\w_1(y)=z;~\w_2(z)=z.$$
\item[$\bullet$] A basis for $H^2_*(\h,z^*)$ is given by $\{(\varphi_1,
0), (\varphi_2,0), (0,\w_1), (0,\w_2)\}$, with
$$\varphi_1(x,y)=x;~\varphi_2(x,y)=y;~
\w_1(y)=z;~\w_2(x)=z.$$
\end{itemize}
\end{thm}
\begin{proof}

With Lemma \ref{hrestcob}, we deduce that $\{(\varphi_6,0),(\varphi_7,0),(\varphi_8,0)\}$ is a basis of $B^2_*(\h,0)$, with
$$\varphi_6(x,y)=x,~\varphi_6(y,z)=z;~~\varphi_7(x,y)=y,~\varphi_7(x,z)=-z;~~\varphi(x,y)=z.$$
Using Lemma \ref{hrescocy}, it is not difficult to complete the above basis to form a basis for $Z_*^2(\h,0)$ and therefore to find the basis of $H^2_*(\h,0)$. The two other cases with $\theta=x^*$ or $\theta=z^*$ are similar.

\end{proof}
\subsubsection{The case $p=3$}

Let $\F$ be a field of characteristic $3$ and let $\varphi\in C_{CE}^2(\h,\h)$. Then, a map $\w:\h\rightarrow\h$ has the $(*)$-property with respect to $\varphi$ if and only if
\begin{equation}\label{3st}
\w(u+v)=\w(u)+\w(v)+2\left(\varphi([u,v],u)+[\varphi(u,v),u]\right)+\varphi([u,v],v)+[\varphi(u,v),v],~u,v\in \h.
\end{equation}

Let $\theta$ be a linear form on $\h$ and let $(\varphi,\w)\in C_{*}^2(\h,\theta)$. We recall that $\h$ is endowed with a $3$-map $(\cdot)^{[3]}$ given by $\theta$ (see Proposition \ref{classifheisenberg}). Since $p=3$, we have
\begin{equation}\label{3ind}
    \ind^2(\varphi,\w)(u,v)=\varphi\left(u,v^{[3]}\right)-\left[\varphi([u,v],v),v\right]+[u,\w(v)].
\end{equation}

\begin{lem}\label{3hrescocy}
The restricted $2$-cocycles for $(\h,\theta)$ are given by pairs $(\varphi,\w)$, where
\begin{itemize}
    \item[$\bullet$] \underline{Case $\theta=0$}:
  
        \begin{equation}\begin{cases}
            \varphi(x,y)&=ax+by+cz\\
            \varphi(x,z)&=dx+ey+fz\\
            \varphi(y,z)&=gx-dy+iz\\
        \end{cases}
        ~~~~~~~~~~
        \begin{cases}
            \w(x)&=-ex+\gamma z\\
            \w(y)&=dy+\epsilon z\\
            \w(z)&=\kappa z\\
        \end{cases}\end{equation}

    \item[$\bullet$] \underline{Case $\theta=x^*$}: same as Lemma \ref{hrescocy};

    \item[$\bullet$] \underline{Case $\theta=z^*$}: same as Lemma \ref{hrescocy}; 

where all the parameters $a,b,c,d,e,f,h,i,\gamma,\epsilon,\kappa$ belong to $\F$.
\end{itemize}
\end{lem}

\begin{proof}
    Similar to Lemma \ref{hrescocy} proof, by using Equation \ref{3ind}.
\end{proof}

A similar computation shows that the restricted $2$-coboundaries are the same as in Lemma \ref{hrestcob}. Nevertheless, one should take care that the $(*)$-property is no longer trivial if $p=3$, but given by Equation $(\ref{3st})$.

\begin{thm}
We have $\dim_{\F}\left(H^2_*(\h,0)\right)=8$ and $\dim_{\F}\left(H^2_*(\h,x^*)\right)=\dim_{\F}\left(H^2_*(\h,z^*)\right)=4.$
\begin{itemize}
\item[$\bullet$] A basis for $H^2_*(\h,0)$ is given by $\{(\varphi_1,
\w_1), (\varphi_2,\w_2), (\varphi_3,0), (\varphi_4,0), (\varphi_5,0), (0,\w_3), (0,\w_4), (0,\w_5)\}$, with
\begin{align*}
\varphi_1(x,z)&=-\varphi_1(y,z)=x,~\w_1(y)=x;~\varphi_2(x,z)=y,~\w_2(x)=x;~\varphi_3(y,z)=z;\\
\varphi_4(x,z)&=z;~\varphi_5(y,z)=y;~\w_3(x)=\w_4(y)=\w_5(z)=z.
\end{align*}
(We only write non-zero identities).
\item[$\bullet$] A basis for $H^2_*(\h,x^*)$ is given by $\{(\varphi_1,
0), (\varphi_2,0), (0,\w_1), (0,\w_2)\}$, with
$$\varphi_1(x,y)=x;~\varphi_2(x,y)=y;~
\w_1(y)=z;~\w_2(z)=z.$$
\item[$\bullet$] A basis for $H^2_*(\h,z^*)$ is given by $\{(\varphi_1,
0), (\varphi_2,0), (0,\w_1), (0,\w_2)\}$, with
$$\varphi_1(x,y)=x;~\varphi_2(x,y)=y;~
\w_1(y)=z;~\w_2(x)=z.$$
\end{itemize}
\end{thm}
\subsection{Example of a Lie-Rinehart structure on restricted Heisenberg algebras}

  Let $p>3$, $L=(\h,z^*)$, let $A$ be the associative algebra spanned by elements $e_1$ and $e_2$, $e_1$ being the unit and endowed with the multiplication $e_2e_2=0$. Suppose that $A$ acts on $L$ trivially ($e_1\cdot u=u,~e_2\cdot u=0~\forall u\in \h$). Then, we can define an anchor map on the pair $(A,L)$ by setting $\rho_x(e_2)=\rho_y(e_2)=0,~\rho_z(e_2)=\gamma e_2,$ $\gamma\in \F$ being either zero, or satisfying $\gamma^{p-1}=1$. We therefore obtain two restricted Lie-Rinehart structures, that we denote by $\mathcal{L}_0$ if $\gamma=0$ and by $\mathcal{L}_{\gamma}$ if $\gamma^{p-1}=1$.

\bigskip

  Let us choose the restricted multiderivation $(m_1,\w_1)$ given by $m_1(x,y)=x,~m_1(x,z)=m_1(y,z)=0,~\w_1(x)=\w_1(z)=0,~\w_1(y)=z.$ It is not difficult to see that $\sigma_1\equiv 0$ is the only suitable symbol map related to $(m_1,\w_1)$. We obtain the infinitesimal deformation
  \begin{equation}
    m_t(x,y)=z+tx,~m_t(x,z)=m_t(y,z)=0,~~\w_t(x)=0,~ \w_t(y)=tz,~\w_t(z)=z.  
  \end{equation}

\begin{prop}
    The deformation $(m_t,\w_t)$ is a deformation in the sense of  Definition \ref{defidefop} if and only if $\gamma=0$. Otherwise, we have a weak deformation.
\end{prop}

\begin{proof}
The  class of $(m_1,\w_1)$ is non-trivial in $H_2^*(\h,\theta)$, therefore Equations (\ref{jacomulti}) and (\ref{pmapmulti}) hold, which means we have a weak deformation. Equations (\ref{condenplus1}) become
\begin{align*}
    \rho\left(u^{[p]}\right)(a)&=\rho^p(u)(a),~~u\in \h,~a\in A,\\
    \rho(\w_1(u))(a)&=\sum_{k=0}^p\rho(u)^k\circ\sigma_1(u)\circ\rho(u)^{p-k}(a)=0.
\end{align*}  
The first equation is always true and the second is true if and only if $\gamma=0$.
 Equations (\ref{condenplus2}) are trivial.
                
\end{proof}

Now, we consider the Lie-Rinehart algebra $\mathcal{L}_0$ with its infinitesimal deformation given by $(m_1,\w_1)$ and  compute the obstruction cochain. For all  $u,v,w\in \h$, we have
\begin{align}
    \obs^{(1)}(u,v,w)&=m_1(u,m_1(v,w))+m_1(v,m_1(w,u))+m_1(w,m_1(u,v)),\\
    \obs^{(2)}(u,v)&=-m_1\left(u,\w_1(v)\right).
\end{align}
It is straightforward to verify that both those maps vanish on the basis of $\h$. Therefore, we have to find $(m_2,\w_2)$ such that $d_*^2(m_2,\w_2)=(0,0)$ to extend the deformation, which is always possible by seeking into $Z_*^2(\h,z^*)$.

\section{Restricted Lie-Rinehart algebras in characteristic 2}
\subsection{Restricted Lie algebras in characteristic 2}

From now, $\F$ denotes a field of characteristic $2$. We provide in the sequel  a cohomology different from that of Fuchs and Evans (\cite{EF08}). However, both cohomologies seem to coincide in low degrees.

\subsubsection{Definitions}
	
In characteristic $2$,  Definition \ref{restdefi} of a restricted Lie algebra reduces to the following one:
	
\begin{defi}\label{restlie2}
	A \textbf{restricted Lie algebra} in characteristic $2$ is a Lie algebra $L$ endowed with a map $(\cdot)^{[2]}:L\longrightarrow L$ such that
	\begin{enumerate}
		\item $(\lambda x)^{[2]}=\lambda^{2}x^{[2]}$, $x\in L$, $\lambda\in \F$;
		\item $\left[x,y^{[2]} \right]=[[x,y],y],~x,y\in L;$
		\item $(x+y)^{[2]}=x^{[2]}+y^{[2]}+[x,y],~x,y\in L$.
	\end{enumerate}
	
\end{defi}

\begin{prop}
	Let $L$ be a restricted Lie algebra in characteristic $2$. 
    
\begin{itemize}
    \item Let $x_1,\cdots,x_n\in L$. Then we have the formula $$\left(\sum_{i=1}^{n}x_i\right)^{[2]}=\sum_{i=1}^{n}x_i^{[2]}+\sum_{1\leq i<j\leq n}[x_i,x_j]. $$
    \item Suppose that the adjoint representation on $L$ is faithful. Then, Conditions 1. and 3. of Definition \ref{restlie2} follow from Condition 2.
\end{itemize}

\end{prop}
	
\begin{proof} The first point follows from a straightforward computation. Let $x,y,z\in L$ and
suppose that the adjoint representation $\ad:x\mapsto \ad_x=[x,\cdot]$ is faithful. If $\lambda\in\F$, we have
$$\ad_{(\lambda x)^{[2]}}(y)=[(\lambda x)^{[2]},y]=[\lambda x,[\lambda x,y]]=\lambda^2[x,[x,y]]=\lambda^2[x^{[2]},y]=\lambda^2\ad_{x^{[2]}}(y).$$ Therefore, we have $(\lambda x)^{[2]}=\lambda^{2}x^{[2]}$. Then,
\begin{align*}
    \ad_{(x+y)^{[2]}}(z)&=[x+y,[x+y,z]]\\
                        &=[x,[x,z]]+[y,[y,z]]+[x,[y,z]]+[y,[x,z]]\\
                        &=[x^{[2]},z]+[y^{[2]},z]+[[x,y],z]\\
                        &=\ad_{x^{[2]}}(z)+\ad_{y^{[2]}}(z)+\ad_{[x,y]}(z).
\end{align*}
It follows that $(x+y)^{[2]}=x^{[2]}+y^{[2]}+[x,y]$.
\end{proof}

%
%
%

\subsubsection{$2$-mappings versus formal power series}

If $\left( L,[\cdot,\cdot],(\cdot)^{[2]}\right)$ is a restricted Lie algebra, it is straightforward to extend the Lie bracket on $L[[t]]$, with the formula
\begin{equation}\label{extbracket}
\left[\sum_{i\geq 0}t^ix_i,\sum_{j\geq 0}t^jx_j \right]=\sum_{i,j}t^{i+j}[x_i,y_j],~~ x_i,y_j\in L. 
\end{equation}

It is clear that $L[[t]]$ is endowed with a Lie algebra structure with this bracket. Now we aim to have a similar formula for the $[2]$-mapping: can we extend the map $(\cdot)^{[2]}$ on $L[[t]]$ in such a way that the latter space is endowed with a restricted Lie algebra structure with respect to the extended bracket?
Let $x_i\in L$ and $\lambda\in\F$. We then have
\begin{equation}\label{scalarformula}
\left(\sum_{i=0}^{n}\lambda^ix_i\right)^{[2]}=\sum_{i=0}^{n}\lambda^{2i}x_i^{[2]}+\sum_{0\leq i<j\leq n}\lambda^{i+j}[x_i,x_j].
\end{equation}

Equation (\ref{scalarformula}) can be proven easily by induction:
if $x_0,\cdots,x_{n+1}\in L$ and $\lambda\in \F$,
	\begin{itemize}
		\item [$\bullet$] $(x_0+\lambda x_1)^{[2]}=x_0^{[2]}+(\lambda x_1)^{[2]}+[x_0,\lambda x_1]=x_0^{[2]}+\lambda^2x_1^{[2]}+\lambda[x_0,x_1]$, so the assertion is true for $n=1$.
		\item [$\bullet$] If the formula is true for a $n\in \N$, we compute by induction:
		\begin{align*}
		\left(\sum_{i=0}^{n+1}\lambda^ix_i \right)^{[2]}&=\left(\sum_{i=0}^{n}\lambda^ix_i +\lambda^{n+1}x_{n+1} \right)^{[2]}\\
		&=\left(\sum_{i=0}^{n}\lambda^ix_i \right)^{[2]}+\lambda^{2(n+1)}x_{n+1}^{[2]}+\left[\sum_{i=0}^{n}\lambda^ix_i,\lambda^{n+1}x_{n+1} \right]\\
		&=\sum_{i=0}^{n}\lambda^{2i}x_i^{[2]}+\lambda^{2(n+1)}x_{n+1}^{[2]}+\sum_{0\leq i<j\leq n}\lambda^{i+j}[x_i,x_j]+\sum_{i=0}^{n}\lambda^{i+n+1}[x_i,x_{n+1}]\\
		&=\sum_{i=0}^{n+1}\lambda^{2i}x_i^{[2]}+\sum_{0\leq i<j\leq n+1}\lambda^{i+j}[x_i,x_j].\\
		\end{align*}
	\end{itemize}

Equation (\ref{scalarformula}) being true, we aim to use it to define a $2$-mapping on the formal power series with one parameter $t$.


\begin{prop}
If $L$ is a Lie algebra over  $\F$, then $L[[t]]$ is a restricted Lie algebra on $\F$ with the extended bracket (\ref{extbracket}) and the $2$-mapping given by 

\begin{equation}\label{extend2map}\left(\sum_{i\geq 0}t^ix_i \right)^{[2]_{t}}:=\sum_{i\geq 0}t^{2i}x_i^{[2]}+\sum_{i,j}t^{i+j}[x_i,x_j].    \end{equation}
\end{prop}
\begin{proof}

We will check the three conditions of  Definition \ref{restlie2}. Let $\lambda\in \F$ and $x_i\in L$.	

\begin{enumerate}	    
\item	    \begin{align*}
	        \left(\lambda\sum_it^ix_i\right)^{[2]_{t}}&=\left(\sum_it^i(\lambda x_i)\right)^{[2]_{t}}\\
	        &=\sum_it^{2i}(\lambda x_i)^{[2]}+\sum_{i<j}t^{i+j}[\lambda x_i,\lambda x_j]\\
	        &=\lambda^2\sum_it^{2i}x_i^{[2]}+\lambda^2\sum_{i<j}t^{i+j}[x_i,x_j]\\
	        &=\lambda^2\left(\sum_it^i x_i\right)^{[2]_{t}}.
	    \end{align*}
	    
\item \begin{align*}
    \left[\sum_it^ix_i,\left(\sum_jt^j y_j\right)^{[2]_t} \right]&=\left[\sum_it^ix_i,\sum_jt^{2j}y_j^{[2]}\right]+\left[\sum_it^ix_i,\sum_{j<k}t^{j+k}[y_j,y_k]\right]\\
    &=\sum_{i,j}t^{i+2j}\left[x_i,y_j^{[2]}\right]+\sum_{\underset{j<k}{i,j,k}}t^{i+j+k}[x_i,[y_j,y_k]]\\
    &=\sum_{i,j}t^{i+2j}\left[x_i,y_j^{[2]}\right]+\sum_{\underset{j<k}{i,j,k}}t^{i+j+k}[y_j,[y_k,x_i]]+\sum_{\underset{j<k}{i,j,k}}t^{i+j+k}[y_k,[x_i,y_j]]\\
    &=\sum_{i,j}t^{i+2j}\left[x_i,y_j^{[2]}\right]+\sum_{\underset{j<k}{i,j,k}}t^{i+j+k}[[x_i,y_k],y_j]+\sum_{\underset{j<k}{i,j,k}}t^{i+j+k}[[x_i,y_j],y_k]\\
    &=\sum_{i,j}t^{i+2j}\left[x_i,y_j^{[2]}\right]+\sum_{\underset{j>k}{i,j,k}}t^{i+j+k}[[x_i,y_j],y_k]+\sum_{\underset{j<k}{i,j,k}}t^{i+j+k}[[x_i,y_j],y_k]\\
    &=\sum_{i,j}t^{i+j+j}[[x_i,y_j],y_j]+\sum_{\underset{j\neq k}{i,j,k}}t^{i+j+k}[[x_i,y_j],y_k]\\
    &=\sum_{i,j,k}t^{i+j+k}[[x_i,y_j],y_k]\\
    &=\left[\left[\sum_it^ix_i,\sum_jt^jy_j\right],\sum_jt^jy_j\right].
\end{align*}

\item The following computation will be useful.

\begin{align}
\nonumber\sum_{i\neq j}t^{i+j}[x_i,y_j]&=\sum_{i<j}t^{i+j}[x_i,y_j]+\sum_{j<i}t^{i+j}[x_i,y_j]\\
\nonumber&=\sum_{i<j}t^{i+j}[x_i,y_j]+\sum_{i<j}t^{i+j}[x_j,y_i]\\
&=\sum_{i<j}t^{i+j}[x_i,y_j]+\sum_{i<j}t^{i+j}[y_i,x_j].\label{eqqq}
\end{align}
Now we can prove the third condition.

\begin{align*}
    \left(\sum_it^ix_i+\sum_j t^j y_j\right)^{[2]_{t}}&=\left(\sum_it^i(x_i+y_i)\right)^{[2]_{t}}\\
    &=\sum_it^{2i}(x_i+y_i)^{[2]}+\sum_{i<j}t^{i+j}[x_i+y_i,x_j+y_j]\\
    &=\sum_it^{2i}x_i^{[2]}+\sum_it^{2i}y_i^{[2]}+\sum_it^{2i}[x_i,y_i]\\
    &~~+\sum_{i<j}t^{i+j}[x_i,x_j]+\sum_{i<j}t^{i+j}[x_i,y_j]\\
    &~~+\sum_{i<j}t^{i+j}[x_j,y_i]+\sum_{i<j}t^{i+j}[x_j,y_j]\\
    &=\sum_it^{2i}x_i^{[2]}+\sum_it^{2i}y_i^{[2]}+\sum_it^{2i}[x_i,y_i]\\
    &~~+\sum_{i<j}t^{i+j}[x_i,x_j]+\sum_{i<j}t^{i+j}[x_j,y_j]\\
    &~~+\sum_{i\neq j}t^{i+j}[x_i,y_j]~~\text{~~~~(with Equation (\ref{eqqq}))}\\
    &=\left(\sum_it^ix_i\right)^{[2]_{t}}+\left(\sum_it^iy_i\right)^{[2]_{t}}+\sum_{i,j}t^{i+j}[x_i,y_j]\\
    &=\left(\sum_it^ix_i\right)^{[2]_{t}}+\left(\sum_it^iy_i\right)^{[2]_{t}}+\left[\sum_it^ix_i,\sum_j t^j y_j\right].\\
\end{align*}
\end{enumerate}		
\end{proof}

\begin{rmq} By expanding the formula (\ref{extend2map}) and by arranging the terms by monomials of the same degree, we obtain
	\begin{equation}
	\left(\sum_{n\geq 0}t^nx_n \right)^{[2]_{t}}=\sum_{n\geq 0}t^n\left( (n+1)x^{[2]}_{\lfloor\frac{n}{2}\rfloor}+\sum_{\underset{i+j=n}{i<j}}[x_i,x_j] \right), 
	\end{equation}
	where $\lfloor\cdot\rfloor$ denotes the floor function.
\end{rmq}


\subsubsection{Cohomology of restricted Lie algebras in characteristic $2$}
Let $M$ be a restricted $L$-module. We start by setting $C_{*_2}^0(L,M)=C_{CE}^0 (L,M)$ and $C_{*_2}^1(L,M)=C_{CE}^1 (L,M)$.

\begin{defi}
	Let $n\geq2$ and $\varphi\in C_{CE}^n (L,M)$ and $\omega:L^{n-1}\longrightarrow M$. The pair $(\varphi,\omega)$ is a $n$-cochain of the restricted cohomology if	
	\begin{align}
	\omega(\lambda x, z_2,\cdots,z_{n-1})&=\lambda^2\omega(x,z_2,\cdots,z_{n-1})\\
	\omega(x,z_2,\cdots,\lambda z_i+z_i',\cdots,z_{n-1})&=\lambda\omega(x,z_2,\cdots,z_i,\cdots,z_{n-1})+\omega(x,z_2,\cdots,z_i',\cdots,z_{n-1})\\
	\omega(x+y,z_2,\cdots,z_{n-1})&=\omega(x,z_2,\cdots,z_{n-1})+\omega(y,z_2,\cdots,z_{n-1})+\varphi(x,y,z_2,\cdots,z_{n-1}).
	\end{align}
We denote the space of $n$-cochains of $L$ with values in $M$ by $C_{*_2}^n(L,M).$	
\end{defi}
We aim to construct coboundary maps 
$d^n_{*_2}:C^n_{*_2}(L,M)\rightarrow C^{n+1}_{*_2}(L,M)$. 
We set for 
$n\geq 2$,  $d^n_{*_2}(\varphi,\omega)=\left(d^n_{CE}(\varphi),\delta^n(\omega)\right),$ with
\begin{align*}
\delta^n\omega(x,z_2,\cdots,z_n)&=x\cdot\varphi(x,z_2,\cdots,z_n)\\
&+\sum_{i=2}^{n}z_i\cdot\omega(x,z_2,\cdots,\hat{z_i},\cdots,z_n)\\
&+\varphi(x^{[2]},z_2,\cdots,z_n)\\
&+\sum_{i=2}^{n}\varphi\left([x,z_i],x,z_2,\cdots,\hat{z_i},\cdots,z_n \right)\\
&+\sum_{1\leq i<j\leq n}\omega\left(x,[z_i,z_j],z_2,\cdots,\hat{z_i},\cdots,\hat{z_j},\cdots,z_n  \right).\\ 
\end{align*}

\begin{lem}
	Let $n\geq2$ and $(\varphi,\omega)\in C^n_{*_2}(L,M)$. Then $\left(d^n_{CE}(\varphi),\delta^n(\omega)\right)\in C^{n+1}_{*_2}(L,M)$. 
\end{lem}

\begin{proof}
	The only difficult issue is to show that \begin{equation}\label{omega*ppty}\delta^n\omega(x+y,z_2,\cdots,z_{n-1})=\delta^n\omega(x,z_2,\cdots,z_{n-1})+\delta^n\omega(y,z_2,\cdots,z_{n-1})+d_{CE}^n\varphi(x,y,z_2,\cdots,z_{n-1}),	 \end{equation} for all  $x,y,z_2,\cdots,z_{n+1}\in L$. We compute:
\begin{align*}
\delta^n\omega(x+y,z_2,\cdots,z_n)&=x\cdot\varphi(x+y,z_2,\cdots,z_n)
+\sum_{i=2}^{n}z_i\cdot\omega(x+y,z_2,\cdots,\hat{z_i},\cdots,z_n)\\
&~~+\varphi((x+y)^{[2]},z_2,\cdots,z_n)
+\sum_{i=2}^{n}\varphi\left([x+y,z_i],x+y,z_2,\cdots,\hat{z_i},\cdots,z_n \right)\\
&~~+\sum_{1\leq i<j\leq n}\omega\left(x+y,[z_i,z_j],z_2,\cdots,\hat{z_i},\cdots,\hat{z_j},\cdots,z_n  \right).\\ 
&=~x\cdot\varphi(x,z_2,\cdots,z_n)+x\cdot\varphi(y,z_2,\cdots,z_n)+y\cdot\varphi(x,z_2,\cdots,z_n)+y\cdot\varphi(y,z_2,\cdots,z_n)\\
&~~+\sum_{i=2}^{n}z_i\cdot\omega(x,z_2,\cdots,\hat{z_i},\cdots,z_n)+\sum_{i=2}^{n}z_i\cdot\omega(y,z_2,\cdots,\hat{z_i},\cdots,z_n)\\
&~~+\sum_{i=1}^{n}\varphi(x,y,z_2,\cdots,\hat{z_i},\cdots,z_n)\\
&~~+\varphi(x^{[2]},z_2,\cdots,z_n)+\varphi(y^{[2]},z_2,\cdots,z_n)+\varphi([x,y],z_2,\cdots,z_n)\\
&~~+\sum_{i=2}^{n}\varphi\left([x,z_i],x,z_2,\cdots,\hat{z_i},\cdots,z_n \right)+\sum_{i=2}^{n}\varphi\left([x,z_i],y,z_2,\cdots,\hat{z_i},\cdots,z_n \right)\\
&~~+\sum_{i=2}^{n}\varphi\left([y,z_i],y,z_2,\cdots,\hat{z_i},\cdots,z_n \right)+\sum_{i=2}^{n}\varphi\left([y,z_i],x,z_2,\cdots,\hat{z_i},\cdots,z_n \right)\\
&~~+\sum_{2\leq i<j\leq n}\omega\left(x,[z_i,z_j],z_2,\cdots,\hat{z_i},\cdots,\hat{z_j},\cdots,z_n \right)\\
&~~+\sum_{2\leq i<j\leq n}\omega\left(y,[z_i,z_j],z_2,\cdots,\hat{z_i},\cdots,\hat{z_j},\cdots,z_n \right)\\
&~~+\sum_{2\leq i<j\leq n}\varphi(x,z_i,z_j,z_2,\cdots,\hat{z_i},\cdots,\hat{z_j},\cdots,z_n).
\end{align*}
We can now identify the desired terms in the above development:

\begin{align*}	
	\delta^n\omega(x,z_2,\cdots,z_n)&=x\cdot\varphi(x,z_2,\cdots,z_n)+\sum_{i=2}^{n}z_i\cdot\omega(x,z_2,\cdots,\hat{z_i},\cdots,z_n)+\varphi(x^{[2]},z_2,\cdots,z_n)\\
	&+\sum_{i=2}^{n}\varphi\left([x,z_i],x,z_2,\cdots,\hat{z_i},\cdots,z_n \right)+\sum_{2\leq i<j\leq n}\omega\left(x,[z_i,z_j],z_2,\cdots,\hat{z_i},\cdots,\hat{z_j},\cdots,z_n \right);\\
\delta^n\omega(y,z_2,\cdots,z_n)&=y\cdot\varphi(y,z_2,\cdots,z_n)+\sum_{i=2}^{n}z_i\cdot\omega(y,z_2,\cdots,\hat{z_i},\cdots,z_n)+\varphi(y^{[2]},z_2,\cdots,z_n)\\
&+\sum_{i=2}^{n}\varphi\left([y,z_i],y,z_2,\cdots,\hat{z_i},\cdots,z_n \right)+\sum_{2\leq i<j\leq n}\omega\left(y,[z_i,z_j],z_2,\cdots,\hat{z_i},\cdots,\hat{z_j},\cdots,z_n \right);\\
d^n_{CE}\varphi(x,y,z_2,\cdots,z_n)
&=x\cdot\varphi(y,z_2,\cdots,z_n)+y\cdot\varphi(x,z_2,\cdots,z_n)+\sum_{i=1}^{n}\varphi(x,y,z_2,\cdots,\hat{z_i},\cdots,z_n)\\
&+\sum_{i=1}^{n}\varphi([x,y],z_2,\cdots,z_n)+\sum_{i=2}^{n}\varphi\left([x,z_i],y,z_2,\cdots,\hat{z_i},\cdots,z_n \right)\\
&+\sum_{i=2}^{n}\varphi\left([y,z_i],x,z_2,\cdots,\hat{z_i},\cdots,z_n \right)+\sum_{2\leq i<j\leq n}\varphi(x,z_i,z_j,z_2,\cdots,\hat{z_i},\cdots,\hat{z_j},\cdots,z_n).\\
\end{align*}
Equation (\ref{omega*ppty}) is then satisfied.
\end{proof}

\begin{lem}
	With the previous data, we have $\delta^{n+1}\circ\delta^n=0$.
\end{lem}

\begin{proof}
\begin{align*}
    \delta^{n+1}\circ\delta^n\w(x,z_2,\cdots,z_{n+1})&=x\cdot d_{CE}^n\varphi(x,z_2,\cdots,z_{n+1})\\
    &~~+\sum_{i=2}^{n+1}z_i\cdot\delta^n\w(x,z_2,\cdots,\hat{z_i},\cdots,z_{n+1})\\
    &~~+d_{CE}^n\varphi(x^{[2]},z_2,\cdots,z_{n+1}\\
    &~~+\sum_{i=2}^{n+1}d_{CE}^n\varphi([x,z_i],x,z_2,\cdots,\hat{z_i},\cdots,z_{n+1})\\
    &~~+\sum_{2\leq i<j\leq n}\delta^n\w(x,[z_i,z_j],z_2,\cdots,\hat{z_i},\cdots,\hat{z_j},\cdots,z_{n+1})\\
    &=\sum_{i=2}^{n+1}z_i\cdot\left( x\cdot\varphi(x,z_2,\cdots,\hat{z_i},\cdots,z_{n+1})\right)\\
    &~~+\sum_{i=2}^{n+1} z_i\cdot\sum_{\underset{j\neq i}{j=2}}z_j\cdot\w(x,z_2,\cdots,\hat{z_i},\cdots\hat{z_j},\cdots,z_{n+1})\\
    &~~+\sum_{i=2}^{n+1}z_i\cdot\varphi(x^{[2]},z_2,\cdots,\hat{z_i},\cdots,z_{n+1})\\
    &~~+\sum_{i=2}^{n+1}z_i\cdot\sum_{\underset{j\neq i}{j=2}}^{n+1}\varphi([x,z_j],x,z_2,\cdots,\hat{z_i},\cdots,\hat{z_j},\cdots,z_{n+1})\\
    &~~+\sum_{i=2}^{n+1}z_i\cdot\sum_{\underset{j,k\neq i}{2\leq j<k\leq n+1}}\w(x,[z_j,z_k],z_2,\cdots,\hat{z_i},\cdots,\hat{z_j},\cdots,\hat{z_k}\cdots,z_{n+1})\\
    &~~+\sum_{i=2}^{n+1}z_i\cdot\varphi(x^{[2]},z_2,\cdots,z_{n+1})+x^{[2]}\cdot\varphi(z_2,\cdots,z_{n+1})\\
    &~~+\sum_{2\leq i<j\leq n+1}\varphi([z_i,z_j],x^{[2]},\cdots,\hat{z_i},\cdots,\hat{z_j},\cdots,z_{n+1})\\
    &~~+\sum_{j=2}^{n+1}\varphi([x^{[2]},z_j],z_2,\cdots,\hat{z_j},...,z_{n+1})\\
    &~~+x\cdot\sum_{i=2}^{n+1}z_i\cdot\varphi(x,z_2,\cdots,\hat{z_i},\cdots,z_{n+1})+x\cdot(x\cdot\varphi(z_2,\cdots,z_{n+1}))\\
    &~~+x\cdot\sum_{2\leq i<j\leq n+1}\varphi([z_i,z_j],x,\cdots,\hat{z_i},\cdots,\hat{z_j},\cdots,z_{n+1})\\
    &~~+x\cdot\sum_{j=2}^{n+1}\varphi([x,z_j],z_2,\cdots,\hat{z_j},\cdots,z_{n+1})\\
    &~~+\sum_{i=2}^{n+1}\sum_{\underset{j\neq i}{j=2}}^{n+1}z_j\cdot\varphi([x,z_i],x,z_2,\cdots,\hat{z_i},\cdots,\hat{z_j},\cdots,z_{n+1})\\
    &~~+\sum_{i=2}^{n+1}x\cdot\varphi([x,z_i],z_2,\cdots,\hat{z_i},\cdots,z_{n+1})\\
    &~~+\sum_{i=2}^{n+1}[x,z_i]\cdot\varphi(x,z_2,\cdots,\hat{z_i},\cdots,z_{n+1})\\
    &~~+\sum_{i=2}^{n+1}\sum_{\underset{j,k\neq i}{2\leq j<k\leq n+1}}\varphi([z_j,z_k],[x,z_i],x,z_2,\cdots,\hat{z_i},\cdots,\hat{z_j},\cdots,\hat{z_k}\cdots,z_{n+1}))\\
    &~~+\sum_{i=2}^{n+1}\sum_{\underset{j\neq i}{j=2}}^{n+1}\varphi([x,z_j],[x,z_i],\cdots,\hat{z_i},\cdots,\hat{z_j},\cdots,z_{n+1})\\
    &~~+\sum_{i=2}^{n+1}\sum_{\underset{j\neq i}{j=2}}\varphi([[x,z_i],z_j],x,z_2,\cdots,\hat{z_i},\cdots,\hat{z_j},\cdots,z_{n+1})\\
    &~~+\sum_{i=2}^{n+1}\varphi([[x,z_i],x],z_2,\cdots,\hat{z_i},\cdots,z_{n+1})\\
    &~~+\sum_{2\leq i<j\leq n+1}x\cdot\varphi(x,[z_i,z_j],\cdots,\hat{z_i},\cdots,\hat{z_j},\cdots,z_{n+1})\\
    &~~+\sum_{2\leq i<j\leq n+1}\sum_{\underset{j,k\neq i}{k=2}}^{n+1}z_k\cdot\w(x,[z_i,z_j]),z_2,\cdots,\hat{z_i},\cdots,\hat{z_j},\cdots,\hat{z_k}\cdots,z_{n+1})\\
    &~~+\sum_{2\leq i<j\leq n+1}[z_i,z_j]\cdot\w(x,z_2\cdots,\hat{z_i},\cdots,\hat{z_j},\cdots,z_{n+1})\\
    &~~+\sum_{2\leq i<j\leq n+1}\varphi(x^{[2]},[z_i,z_j],z_2,\cdots\hat{z_i},\cdots,\hat{z_j},\cdots,z_{n+1})\\
    &~~+\sum_{2\leq i<j\leq n+1}\varphi([x,[z_i,z_j],x,z_2,\cdots\hat{z_i},\cdots,\hat{z_j},\cdots,z_{n+1})\\
    &~~+\sum_{2\leq i<j\leq n+1}\sum_{\underset{j,k\neq i}{k=2}}^{n+1}\varphi([x,z_k],x,[z_i,z_j],z_2,\cdots,\hat{z_i},\cdots,\hat{z_j},\cdots,\hat{z_k}\cdots,z_{n+1})\\
    &~~+\sum_{2\leq i<j\leq n+1}~~\sum_{\underset{k,l\neq i,j}{2\leq k<l\leq n+1}}\w(x,[z_k,z_l],z_2,\cdots,\hat{z_i},\cdots,\hat{z_j},\cdots,\hat{z_k}\cdots,\hat{z_l},\cdots,z_{n+1})\\
    &~~+\sum_{2\leq i<j\leq n+1}^{n+1}\sum_{\underset{k\neq i,j}{k=2}}^{n+1}\w(x,[[z_i,z_j],z_k],z_2,\cdots,\hat{z_i},\cdots,\hat{z_j},\cdots,\hat{z_k}\cdots,z_{n+1})\\
    &=0.
\end{align*}
\end{proof}
Thus, we have obtained a cochain complex $\left(C_{*_2}^n(L,M),d^n_{*_2} \right)_{n\geq 2}$. In particular, for $n\in\left\lbrace 0,1\right\rbrace$, we have $d^0_{*_2}=d^0_{CE}$ and 
\begin{align*}
d^1_{*_2}:C_{*_2}^1(L,M)&\longrightarrow C_{*_2}^2(L,M)\\
	\varphi&\longmapsto(d_{CE}^1\varphi,\omega),~\omega(x)=\varphi\left( x^{[2]} \right)+x\cdot\varphi(x),~x\in L.\\
\end{align*}

\begin{lem}
The map	$d^1_{*_2}$ is well-defined,  $d^1_{*_2}\circ d^0_{*_2}=0 \text{ and } d^2_{*_2}\circ d^1_{*_2}=(0,0).$
\end{lem}


Our cochain complex is now complete.

\begin{thm}
	Let $L$ be a restricted Lie algebra and $M$ a restricted $L$-module. The complex $\left(C_{*_2}^n(L,M),d^n_{*_2} \right)_{n\geq 0}$ is a cochain complex. The \textbf{$n^{th}$ restricted cohomology group of the Lie algebra $L$ in characteristic 2} is defined by
	
	$$ H_{*_2}^n(L,M):=Z_{*_2}^n(L,M)/B^n_{*_2}(L,M),$$ 
	with $Z_{*_2}^n(L,M)=\Ker(d^n_{*_2})$ the restricted $n$-cocycles and $B_{*_2}^n(L,M)=\im(d^{n-1}_{*_2})$ the restricted $n$-coboundaries.
\end{thm}

\begin{rmq} 
$H^0_{*_2}(L,M)=H^0_{CE}(L,M)$.
\end{rmq}
\begin{rmq}
This cohomology has no analog in characteristic different from $2$. Very similar cohomology formulas have been considered in \cite{BM22}, in the slightly different context of Hom-Lie superalgebras (of characteristic $2$).
\end{rmq}

\subsubsection{Particular cases: first and second cohomology groups}

\paragraph{First cohomology group and restricted derivations.}

We recall that a restricted derivation $D$ of a restricted Lie algebra $L$ in characteristic $2$ is a derivation that satisfies $D(x^{[2]})=[x,D(x)]$ for all $x\in L$. Let $\varphi$ be a restricted $1$-cocycle, that reads, for all $x,y\in L$:

\begin{center}
$\begin{cases}
\varphi([x,y])=[x,\varphi(y)]+[x,\varphi(y)];\\ 
~~\varphi(x^{[2]})=[x,\varphi(x)].
\end{cases}$
\end{center}

At this point, it is clear that any $1$-cocyle $\varphi$ with values in $L$ is a restricted derivation and that the converse is true. A quick computation shows that
$$B^1_{*_2}(L,L)=B^1_{CE}(L,L)=\im(d^0_{CE})=\left\lbrace \ad_x,~x\in L \right\rbrace.$$
We have seen that every derivation of the form $\ad_x$ is restricted. Those derivations are called \textbf{inner derivations}. Therefore, we have 
$$ H^1_{*_2}(L,L)=Z^1_{*_2}(L,L)/B^1_{*_2}(L,L)=\left\lbrace\text{restricted derivations} \right\rbrace /\left\lbrace\text{inner derivations} \right\rbrace.$$
We recover a well-known result (\cite{ET00}).

\paragraph{Second scalar cohomology group and central extensions.} Let $\left(L,[\cdot,\cdot]_L,(\cdot)^{[2]_L}\right)$ be a restricted Lie algebra in characteristic $2$. We denote by $\g:=L\oplus \F c$, where $c$ is a given generator. We recall that $\F$ is a trivial $L$-module. A restricted scalar $2$-cocycle is a pair $(\varphi,\w)\in C^2_{*_2}(L,\F)$ satisfying

\begin{equation}
    \varphi(x,[y,z]_L)+\varphi(y,[z,x]_L)+\varphi(z,[x,y]_L)=0
\end{equation}
and
\begin{equation}\label{scalar2cocycle}
    \varphi(x,y^{[2]_L})=\varphi([x,y]_L,y). 
\end{equation}
Let $x,y\in L$ and $u,v\in \F$. We define a bracket on $\g$ by
\begin{equation}
    [x+uc,y+vc]_{\g}=[x,y]_L+\varphi(x,y)c
\end{equation}
and a map $(\cdot)^{[2]_{\g}}:\g\rightarrow\g$ by
\begin{equation}
    (x+uc)^{[2]_{\g}}=x^{[2]_L}+\w(x)c.
\end{equation}

\begin{prop} The tuple 
$\left(\g,[\cdot,\cdot]_{\g},(\cdot)^{[2]_{\g}}\right)$ defines a restricted Lie algebra if and only if $(\varphi,\w)$ is a restricted $2$-cocycle.
\end{prop}

\begin{proof}
It is well-known in the ordinary case that $\left(\g,[\cdot,\cdot]\right)$ is a Lie algebra if and only if $\varphi$ is a Chevalley-Eilenberg $2$-cocycle. It remains to show that $(\cdot)^{[2]_{\g}}$ is a $[2]$-mapping on $\g$ if and only if Equation (\ref{scalar2cocycle}) is satisfied. Let $x,y\in L$ and $u,v\in \F$.
\begin{align*}
    \left((x+u)+(y+v)\right)^{[2]_{\g}}&=(x+y)^{[2]_L}+\w(x+y)c\\
    &=x^{[2]_L}+[x,y]_L+\w(x)c+\w(y)c+\varphi(x,y)c\\
    &=(x+uc)^{[2]_{\g}}+(y+vc)^{[2]_{\g}}+[(x+uc),(y+vc)]_{\g}.
\end{align*}
\begin{align*}
    [(x+uc),(y+vc)^{[2]_{\g}}]_{\g}&=[(x+uc),y^{[2]_L}+\w(y)c]_{\g}\\
    &=\left[[x,y^{[2]_L}\right]_L+\varphi\left(x,y^{[2]_L}\right)c\\
    &=[[x,y]_L,y]_L+\varphi([x,y]_L,y)c\\
    &=[[x,y]_L+\varphi(x,y)c,y+vc]_{\g}\\
    &=[[x+uc,y+vc]_{\g},y+vc]_{\g}.
\end{align*}
Finally we have $(\lambda(x+uc))^{[2]_{\g}}=\lambda^2(x+uc)^{[2]_{\g}},$ we conclude that $(\cdot)^{[2]_{\g}}$ is a $[2]$-mapping on $\g$ if and only if Equation (\ref{scalar2cocycle}) is satisfied.

\end{proof}

\subsection{Deformations of restricted Lie-Rinehart algebras in characteristic 2}
In this section, we recall the definition of  restricted Lie-Rinehart algebras in characteristic 2 and discuss their characterization in terms of restricted multiderivations. Moreover, we study their formal deformations using the cohomology defined in the previous section.
\subsubsection{Restricted Lie-Rinehart algebras in characteristic 2}

\begin{defi}\label{RLR2}
 	Let $A$ be a  commutative associative algebra over a field $\F$ of characteristic $2$. Then $(A,L)$ is a \textbf{restricted Lie-Rinehart algebra} in characteristic 2 if
 	\begin{itemize}
 		\item $(A,L)$ is a Lie-Rinehart algebra, with anchor map $\rho:L\longrightarrow\Der(A)$;
 		\item $\left(L, (-)^{[2]}\right)$ is a restricted Lie algebra; 
 		\item $\rho(x^{[2]})=\rho(x)^2$ ($\rho$ is a restricted Lie morphism);
 		\item $(ax)^{[2]}=a^px^{[2]}+\rho(ax)(a)x,~a\in A,~x\in L$. (Hochschild condition)
 	\end{itemize}
 \end{defi}

\begin{defi}\label{2restmulti}
A \textbf{restricted multiderivation} (of order $1$) is a pair $(m,\w)$, where $m:L\times L\rightarrow L$ is skew-symmetric bilinear map  and $\w : L \rightarrow L$ is a $2$-homogeneous map satisfying
\begin{equation}
    \w(x+y)=\w(x)+\w(y)+m(x,y),
\end{equation}
such that it exists a map $\sigma:L\rightarrow \Der(A)$ called \textbf{restricted symbol map} which must satisfy the following four conditions, for all  $x,y\in L$ and $a\in A$:
\begin{equation}\label{2jesaispas1}
    \sigma(ax)=a\sigma(x);
\end{equation}
\begin{equation}\label{2jesaispas2}
    m(x,ay)=am(x,y)+\sigma(x)(a)y;
\end{equation}
\begin{equation}\label{2jesaispas3}
    \sigma\circ\w(x)=\sigma(x)^2;
\end{equation}
\begin{equation}\label{2jesaispas4}
    \w(ax)=a^2\w(x)+\sigma(ax)(a)x.
\end{equation}
\end{defi}


\begin{prop}
There is a one-to-one correspondence between  restricted Lie-Rinehart structures on the pair $(A,L)$ and restricted multiderivations $(m,\w )$ of order $1$ satisfying 
\begin{equation}\label{2jesaispas5}
    m(x,m(y,z))+m(y,m(z,x))+m(z,m(x,y))=0
\end{equation}
and
\begin{equation}\label{2jesaispas6}
    m(x,\w(y))=m(m(\cdots m(x,\overset{p\text{ terms}}{\overbrace{y),y),\cdots,y}})
\end{equation}
\end{prop}


\subsubsection{Restricted formal deformations}

\begin{defi}\label{defidefo2}
	A restricted formal deformation of $(m,\w)$ is given by two maps
%

\begin{multicols}{5}
	\text{}
	\columnbreak
	\begin{align*}m_t:L\times L&\longrightarrow L[[t]]\\
	(x,y)&\longmapsto \displaystyle\sum_{i\geq 0}t^i m_i(x,y);\\
	\end{align*}
	\columnbreak
	\text{}
	\columnbreak
	\hspace{-3cm}\begin{align*}
	\w_t:L&\longrightarrow L[[t]]\\
	x&\longmapsto \displaystyle\sum_{j\geq 0}t^j\omega_j(x),\\
	\end{align*}
	\columnbreak
	\text{}
\end{multicols}

	with $m_0=m$, $\w_0=\w$ and $(m_i,\w_i)$ restricted multiderivations.
	
	Moreover, $m_t$ and $\w_t$ must satisfy, for all  $x,y,z\in L$,
	
	\begin{equation}\label{eqdef1}
	m_t(x,m_t(y,z))+m_t(y,m_t(z,x))+m_t(z,m_t(x,y))=0;
	\end{equation}
	\begin{equation}\label{eqdef2}
	m_t(x,\w_t(y))=m_t(m_t(x,y),y);
	\end{equation}
	\begin{equation}\label{eqdef3}
	\w_t(x+y)=\w_t(x)+\w_t(y)+m_t(x,y);
	\end{equation}
	\begin{equation}\label{eqdef4}
	\sum_{i=0}^k\sigma_i\left(\w_{k-i}(x)\right)=\sum_{i=0}^k\sigma_i(x)\circ\left(\sigma_{k-i}(x)\right),~~\forall k\geq 0.
	\end{equation}
	
\end{defi}

\begin{rmq}
	\hspace{0.3cm}
	\begin{enumerate}
		\item The map $m_t$ extends to $L[[t]]$ by $\F[[t]]$-linearity.
		
		\item The map $\w_t$ extends to $L[[t]]$ using Equations (\ref{extend2map}) and (\ref{eqdef3}).
	\end{enumerate}
\end{rmq}

\begin{rmq}
	Condition $(\ref{eqdef1})$ ensures that the deformed object $\left(A, L[[t]], m_t,\sigma_t\right)$ is a Lie-Rinehart algebra. Conditions $(\ref{eqdef1})$, $(\ref{eqdef2})$  and $(\ref{eqdef3})$ ensure that $\left( L[[t]], m_t,\w_t\right)$ is a restricted Lie algebra. Moreover, if Condition $(\ref{eqdef4})$ is satisfied, then $\left(A, L[[t]], m_t,\w_t,\sigma_t\right)$ is a restricted Lie-Rinehart algebra.
\end{rmq}

\begin{defi}
    If the condition $(\ref{eqdef4})$ of  Definition $\ref{defidefo2}$ are not satisfied, the deformation is called \textbf{weak deformation}.
\end{defi}

\begin{rmq}
    Here we see that there is no analog of Condition (\ref{condenplus2}) that we imposed in Definition \ref{defidefop}. If $p=2$, the Hochschild condition of  Definition \ref{RLR2} becomes linear in $\w$ and $\sigma$, therefore $\w_t$ and $\sigma_t$ satisfy it without imposing any further condition.
\end{rmq}

\begin{rmq}
    Like in the previous section, all the following results dealing with restricted deformations, equivalences and obstructions in characteristic $2$ are also valid for restricted Lie algebras, by forgetting the Lie-Rinehart structure. Weak deformations of a restricted  Lie-Rinehart algebra correspond to deformations of  restricted Lie algebras.
\end{rmq}

\begin{lem}
	Let $\left(m_t,\w_t \right)$ be a restricted deformation of $(m,\w)$. Then $\left(m_k,\w_k \right)\in C_{*_2}^2(L,L)~\forall k\geq 0$.
\end{lem}

\begin{proof}
	Let $x,y\in L$. By expanding Equation (\ref{eqdef3}), we obtain
	
	$$\sum_{i\geq 0}t^i\omega_i(x+y)=\sum_{i\geq 0}t^i\omega_i(x)+\sum_{i\geq 0}t^i\omega_i(y)+\sum_{i\geq 0}t^im_i(x).$$ Then, for every $k\geq 0$, we have
	$$\omega_k(x+y)=\omega_k(x)+\omega_k(y)+m_k(x,y),$$ which is the desired identity. Moreover, for $\lambda\in \F$, we have
	
	$$ \w_t(\lambda x)=\sum_{i\geq 0}t^i\omega_i(\lambda x)=\lambda^2\sum_{i\geq 0}t^i\omega_i(x), \text{ so } \omega_i(\lambda x)=\lambda^2\omega_i(x),~\forall i\geq0.  $$
\end{proof}
We have the following classical result:

\begin{prop}
Let $\left(m_t,\w_t \right)$ be a restricted deformation of $(m,\w)$. Then $(m_1,\w_1)$ is a $2$-cocycle of the restricted cohomology, that reads
	
	$$ d_{CE}^2m_1=0 \text{ and } \delta^2\w_1=0.    $$
\end{prop}

\begin{proof}
	The ordinary theory ensures that $d^2_{CE}m_1=0$. It remains to check that $\delta^2\w_1=0$. If we expand Equation (\ref{eqdef2}), we obtain
	
	\begin{equation}
	    \sum_{i,j}t^{i+j}m_i(x,\w_j(y))=\sum_{i,j}t^{i+j}m_i(m_j(x,y),y).
	\end{equation}
	By collecting the coefficients of $t$, we obtain
	
	$$m_1(x,\w(y))+m(x,\w_1(y))=m(m_1(x,y),y)+m_1(m(x,y),y),  $$
	which is equivalent to $\delta^2\w_1=0.$
\end{proof}

\subsubsection{Equivalence of restricted formal deformations}

Let $\phi:L[[t]]\longrightarrow L[[t]]$ be a formal automorphism defined on $L$ by 
$$ \phi(x)=\sum_{i\geq 0}t^i\phi_i(x),~\phi_i:L\longrightarrow L,~\phi_0=id,  $$ and then extended by $\F[[t]]$-linearity.

\begin{defi}
	Two formal deformations $\left(m_t,\w_t \right)$ and $\left(m'_t,\w'_t \right)$ of $(m,\w)$ are \textbf{equivalent} if there exists a formal automorphism $\phi_t$ such that
	\begin{eqnarray}\label{eqeq1}
	 m_t\left(\phi_t(x),\phi_t(y) \right)&=&\phi_t\left(m'_t(x,y) \right)\text{  and  }  
	\\ \label{eqeq2}
	  \phi_t\left(\omega_t(x) \right)&=&\w'_t\left(\phi_t(x)\right). 
	\end{eqnarray}

\end{defi}

\begin{lem}
	With the above data, we have
	\begin{eqnarray}
	m_i(x,y)+m'_1(x,y)&=&\phi_1(m(x,y))+m(x,\phi_1(y))+m(\phi_1(x),y);\\
	\omega_1(x)+\w'_1(x)&=&m(\phi_1(x),x)+\phi_1\left(\w(x)\right).
	\end{eqnarray}
\end{lem}

\begin{proof}
	\begin{align*}
	    (\ref{eqeq1})&\Longleftrightarrow m_t\left(\sum_it^i\phi_i,\sum_j t^j\phi_j\right)=\left(\sum_kt^km'_t(x,y)\right)\\
	    &\Longleftrightarrow\sum_{i,j,k}t^{i+j+k}m_k\left(\phi_i(x),\phi_j(y)\right)=\sum_{i,j}t^{i+j}\phi_j(m'_i(x,y)).
	\end{align*}
	By collecting the coefficients of $t$, we obtain
	$$  m(\phi_1(x),y)+m(x,\phi_(y))+m_1(x,y)=\phi_1(m(x,y))+m'_1(x,y),  $$ which is equivalent to the first desired equation. The following computations are made mod $t^2$.
	\begin{align*}
	    (\ref{eqeq2})&\Rightarrow\phi_t\left(\sum_it^i\w_i(x)\right)=\w'_t\left(x+t\phi_1(x)\right)\\
	    &\Rightarrow\sum_it^i\phi_t(\w_i(x))=\w'_t(x)+\w'_t(t\phi_1(x))+m'_t(x,t\phi_1(x))\\
	    &\Rightarrow \sum_{i,j}t^{i+j}\phi_j(w_i(x))=\w(x)+t\left(\w'_1(x)+m(x,\phi_1(x))\right).
	\end{align*}
Considering  coefficient of $t$, one has 
	$$\w_1(x)+\phi_1(\w(x))=\w'_1(x)+m(x,\phi_1(x)),$$ which is equivalent to the second desired equation.
	
\end{proof}

\begin{prop}
	If $\left(m_t,\omega_t \right) $ and $\left( m'_t,\w'_t \right) $ are two equivalent deformations of $(m,\w)$, then their infinitesimals $\left(m_1,\omega_1 \right) $ and $\left( m'_1,\w'_1 \right) $ are cohomologous.
\end{prop}

\begin{proof}
The formula in previous lemma may be written as: 
  $$\phi_1(m(x,y))+m(x,\phi_1(y))+m(m\phi_1(x),y)=d_{CE}^1\phi_1(x,y)  $$ and 
  $$m(\phi_1(x),x)+\phi_1(\w(x))=\delta^1\phi_1.$$
\end{proof}

\begin{defi}
	A formal deformation  $\left(m_t,\omega_t \right)$  of $(m,\w)$ is called \textbf{trivial} if there exists  a formal automorphism $\phi_t$ such that
	\begin{eqnarray}
	\phi_t\left(m_t(x,y) \right)&=&m(\phi_t(x),\phi_t(y));\\
	\phi_t\left(\omega_t(x) \right)&=&\w\left(\phi_t(x) \right).   
	\end{eqnarray}
\end{defi}

\begin{prop}
	Suppose that $(m_1,\omega_1)\in B^2_{*_2}(L,L)$. Then the infinitesimal deformation given by $m_t=m+tm_1$ and $\omega_t=\w+t\omega_1$ is trivial.
\end{prop}

\begin{proof}
	Let $(m_1,\omega_1) \in B^2_{*_2}(L,L)$. Then there exists $ \varphi: L\rightarrow L$ such that $m_1=d_{CE}^2\varphi$ and $\w_1=\delta^1\varphi$. We consider a formal automorphism
	$$\phi_t=\id+t\varphi.  $$
	Since $m_1$ is a Chevalley-Eilenberg $2$-coboundary, we have
	\begin{equation}
	    m_1(x,y)=\varphi(m(x,y))+m(x,\varphi(y))+m(\varphi(x),y).
	\end{equation}
	Thus, we can write
$$m(x,y)+t\left(\varphi(m(x,y))+m_1(x,y) \right)=m(x,y)+t\left(m(x,\varphi(y))+m(\varphi(x),y)\right),$$
	which is equivalent to
$$\phi_t\left(m(x,y)+tm_1(x,y)\right)=m(\phi_t(x),\phi_t(y)).$$
Finally, we obtain (mod $t^2$)
	\begin{equation}\label{prooftriv1}
	    \phi_t(m_t(x,y))=m(\phi_t(x),\phi_t(y)).
	\end{equation}
Then, using the identity $\w_1=\delta^1\varphi$, we have

\begin{equation}
    \w_1(x)+\varphi\left(\w(x)\right)=m(x,\varphi(x)).
\end{equation}
Thus, we can write
$$ \w(x)+t\left(\w_1(x)+\phi\left( \w(x)\right)\right)=\w(x)+tm(x,\varphi(x)),$$
which is equivalent (mod $t^2$) to
$$ \phi_t\left(\w(x)+t\w_1(x)\right)=\w(x+t\varphi(x)).$$
Finally, we obtain (mod $t^2$)
\begin{equation}\label{prooftriv2}
    \phi_t\left(\w_t(x)\right)=\w(\phi_t(x)).
\end{equation}	
Equations (\ref{prooftriv1}) and (\ref{prooftriv2}) together implies that the deformation is trivial.
	
\end{proof}

\begin{thm}
	The second restricted cohomology space $H^2_{*_2}(L,L)$ classifies, up to equivalence, the infinitesimal restricted deformations.
\end{thm}

\subsubsection{Obstructions}

Let $\left(m,\w \right)$ be a restricted multiderivation.
A restricted deformation $(m^n_t,\w^n_t)$ is of order $n\in \N$ if the defining formal power series are truncated up to order $n$, that is
$$m_t^n=\sum_{k=0}^{n}t^km_k~\text{ and }~\w_t^n=\sum_{k=0}^{n}t^k\omega_k.$$

\begin{defi}
	For all  $x,y,z\in L$, we set: 
	\begin{align*}
	\obs_{n+1}^{(1)}(x,y,z)&=\sum_{i=1}^{n}\left(m_i(x,m_{n+1-i}(y,z))+m_i(y,m_{n+1-i}(z,x))+m_i(z,m_{n+1-i}(x,y)) \right);\\ 
	\obs_{n+1}^{(2)}(x,y)&=\sum_{i=1}^{n}\left( m_i(y,\omega_{n+1-i}(x))+m_i(m_{n+1-i}(y,x),x) \right).
	\end{align*}
\end{defi}

\begin{lem}
	$\left( \obs_{n+1}^{(1)},\obs_{n+1}^{(2)}\right)\in C^3_{*_2}(L,L)$. 
\end{lem}

\begin{proof}
	
\begin{align*}
		\obs_{n+1}^{(2)}(x_1+x_2,y)&=\sum_{i=1}^{n}\left( m_i(y,\omega_{n+1-i}(x_1+x_2))+m_i\left( m_{n+1-i}(y,x_1+x_2),x_1+x_2 \right)\right) \\
		&=\sum_{i=1}^{n}\left( m_i\left(y,\omega_{n+1-i}(x_1)\right) + m_i\left(y,\omega_{n+1-i}(x_2) \right)+m_i\left(y,m_{n+1-i}(x_1,x_2)\right) \right)  \\
		&~~+\sum_{i=1}^{n}\left( m_i\left(m_{n+1-i}(y,x_1),x_1 \right) +m_i\left(m_{n+1-i}(y,x_1),x_2 \right)\right) \\
		&~~+\sum_{i=1}^{n}\left( m_i\left(m_{n+1-i}(y,x_2),x_1 \right)+m_i\left(m_{n+1-i}(y,x_2),x_2 \right)\right)  \\
		&=\sum_{i=1}^{n}\left( m_i\left(y,\omega_{n+1-i}(x_1)\right) + m_i\left(m_{n+1-i}(y,x_1),x_1 \right)\right)\\
		&~~+\sum_{i=1}^{n}\left( m_i\left(y,\omega_{n+1-i}(x_2)\right) + m_i\left(m_{n+1-i}(y,x_2),x_2 \right)\right)\\
		&~~+\sum_{i=1}^{n}\left(m_i\left(m_i(y,m_{n+1-i}(x_1,x_2) \right)+m_1\left(m_{n+1-i}(y,x_1),x_2 \right)+m_1\left(m_{n+1-i}(y,x_2),x_1 \right)\right)\\
		&=\obs_{n+1}^{(2)}(x_1,y)+\obs_{n+1}^{(2)}(x_2,y)+\obs_{n+1}^{(1)}(x_1,x_2,y).
\end{align*}	
\end{proof}

\begin{prop}

	Let $\left( m^n_t,\w^n_t \right) $ be a $n$-order deformation of $\left(m,\w \right)$. Let $\left(m_{n+1},~\w_{n+1} \right)\in C^2_{*_2}(L,L)$. Then $\left( m^n_t+t^{n+1}m_{n+1},~  \w^n_t+t^{n+1}\w_{n+1} \right) $ is a $(n+1)$-order deformation of $L$ if and only if  
	$$	\left( \obs^{(1)}_{n+1}, \obs^{(2)}_{n+1}\right)=d_{*_2}^2\left(m_{n+1},~\w_{n+1} \right).        $$

\end{prop}

\begin{proof}
	If $m^n_t+t^{n+1}m_{n+1}  $
	satisfies the Jacobi identity, we obtain for all $x,y,z\in L$:
	 \begin{equation}
	 \sum_{i=0}^{n+1}\left(m_i(x,m_{n+1-i}(y,z)) +m_i(y,m_{n+1-i}(z,x)) +m_i(z,m_{n+1-i}(x,y))\right)=0.
	 \end{equation}
The above  equation can be written	
\begin{equation}
\sum_{i=1}^{n}\left(m_i(x,m_{n+1-i}(y,z)) +m_i(y,m_{n+1-i}(z,x)) +m_i(z,m_{n+1-i}(x,y))\right)=d_{CE}^2m_{n+1}(x,y,z).
\end{equation} 
Conversely, if $\obs^{(1)}_{n+1}=d_{CE}^2m_{n+1}$, then $m^n_t+t^{n+1}m_{n+1}$ satisfies the Jacobi identity. Now suppose that $\w^n_t+t^{n+1}\w_{n+1}$ is a $2$-map with respect to $m^n_t+t^{n+1}m_{n+1}$. The following equation is then satisfied:

\begin{equation}\label{obseq}
m_t^{n+1}\left(x,\w_t^{n+1} \right)=m_t^{n+1}\left( m_t^{n+1}(x,y),y\right),  
\end{equation} 
where we have denoted for convenience $m_t^{n+1}=m^n_t+t^{n+1}m_{n+1}$ and $\w^{n+1}_t=\w^n_t+t^{n+1}\w_{n+1}$.
By expanding Equation (\ref{obseq}), we obtain

\begin{equation}
	\sum_{q=0}^{n+1}t^q\sum_{i=0}^{q}m_i(x,\omega_{q-i}(y))=\sum_{q=0}^{n+1}t^q\sum_{i=0}^{q} m_i(m_{q-i}(x,y),y).
\end{equation}	 
Collecting coefficient of $t^{n+1}$, one obtains in particular
$$
\sum_{i=0}^{n+1}m_i(x,\omega_{n+1-i}(y))=\sum_{i=0}^{n+1}t^q\sum_{i=0}^{n+1} m_i(m_{n+1-i}(x,y),y),
$$
which can be written
\begin{align*}
m(x,\omega_{n+1}(y))+m_{n+1}(x,\w(y))&+m(m_{n+1}(x,y),y)+m_{n+1}(m(x,y),y)\\
&=\sum_{i=1}^{n}\left(m_i(x,\omega_{n+1-i}(y))+ m_i(m_{n+1-i}(x,y),y)\right). 
\end{align*}
We conclude that
$$\sum_{i=1}^{n}\left(m_i(x,\omega_{n+1-i}(y))+ m_i(m_{n+1-i}(x,y),y)\right)=\delta_{*_2}^2\omega_{n+1}(x,y).$$

\end{proof}
\subsection{Restricted Heisenberg algebras in characteristic 2}

In this section, we consider Heisenberg algebras in characteristic 2. We provide all restricted structures on the Heisenberg algebra, compute the adjoint  second cohomology group. Moreover, we consider an example of  restricted Lie-Rinehart algebra based on  a restricted Heisenberg algebra and discuss its deformations. We refer to Section \ref{sectionh} about Heisenberg algebra's background.
\subsubsection{Restricted cohomology}
Let $\F$ be an algebraically closed field of characteristic $2$. Recall that the Heisenberg algebra $\h$ is spanned by elements $x,y,z$ with a bracket given by $[x,y]=z.$ It is worth noticing that in characteristic $2$, $\h$ is isomorphic to $\Sl_2$.    Let $(\cdot)^{[2]}$ be a $2$-mapping on $\h$. Then, we have
\begin{align}
\nonumber(x+y)^{[2]}&=x^{[2]}+y^{[2]}+z;\\
(x+z)^{[2]}&=x^{[2]}+z^{[2]};\\\nonumber
(y+z)^{[2]}&=y^{[2]}+z^{[2]}.\nonumber
\end{align}
Therefore, the $2$-mapping is not $2$-semilinear. Let $u=ax+by+cz\in\h$, $a,b,c\in \F$. Then
\begin{equation}
u^{[2]}=(ax+by+cz)^{[2]}=a^2x^{[2]}+b^2y^{[2]}+c^2z^{[2]}+abz.
\end{equation}

According to  the second condition of the Definition \ref{restlie2}, the image of $(\cdot)^{[2]}$ lies in the center of $\h$, which is one-dimensional and spanned by $z$. Therefore, it exists $\theta:\h\rightarrow \F$ linear such that $x^{[2]}=\theta(x)z,~y^{[2]}=\theta(y)z,~z^{[2]}=\theta(z)z.$ It is not difficult to see that Lemma \ref{heisiso} also holds for $p=2$. We deduce the following classification result, using the same techniques as Theorem \ref{heisclass}.

\begin{thm}
    There are two non-isomorphic restricted Heisenberg algebras in characteristic $2$, respectively given by the linear forms $\theta=0$ and $\theta=z^*$.
\end{thm}

We compute the second restricted cohomology groups of the restricted Heisenberg algebras with adjoint coefficients. Let $\theta$ be a linear form on the (ordinary) Heisenberg algebra. We denote by $(\h,\theta)$ the restricted Heisenberg algebra obtained with $\theta$. We also denote $H^2_*(\h,\theta)=H^2_{*_2}((\h,\theta),(\h,\theta))$ the adjoint second restricted cohomology group of $(\h,\theta)$.

Let $u,v\in\h$ and $(\varphi,\w)\in C_{*_2}(\h,\theta)$. The restricted $2$-cocycle condition is given by
\begin{equation}
    \varphi\left(u,\theta(v)z\right)+[u,\w(v)]+[\varphi(u,v),v]+\varphi([u,v],v)=0.
\end{equation}
The restricted $2$-coboundary condition is given by
\begin{equation}
    \w(u)=[\varphi(u),u]+\varphi\left(\theta(u)z\right).
\end{equation}
Using the same techniques as Lemmas \ref{hrescocy} and  \ref{hrestcob}, we obtain the general form of the $2$-cocycles and $2$-coboundaries.

\begin{lem}\label{2hrescocy}
The restricted $2$-cocycles for $(\h,\theta)$ are given by pairs $(\varphi,\w)$, where
\begin{itemize}
    \item[$\bullet$] \underline{Case $\theta=0$}:
        \begin{equation}\begin{cases}
            \varphi(x,y)&=ax+by+cz\\
            \varphi(x,z)&=fz\\
            \varphi(y,z)&=iz\\
        \end{cases}
        ~~~~~~~~~~
        \begin{cases}
            \w(x)&=(b+f)x+\gamma z\\
            \w(y)&=(a+i)y+\epsilon z\\
            \w(z)&=\kappa z\\
        \end{cases}\end{equation}

    \item[$\bullet$] \underline{Case $\theta=z^*$}:
  
        \begin{equation}\begin{cases}
            \varphi(x,y)&=ax+by+cz\\
            \varphi(x,z)&=fz\\
            \varphi(y,z)&=iz\\
        \end{cases}
        ~~~~~~~~~~
        \begin{cases}
            \w(x)&=(b+f)x+\gamma z\\
            \w(y)&=(a+i)y+\epsilon z\\
            \w(z)&=ix+fy+\kappa z,\\
        \end{cases}\end{equation} where all the parameters $a,b,c,d,e,f,h,i,\gamma,\epsilon,\kappa$ belongs to $\F$.
\end{itemize}
\end{lem}

\begin{lem}\label{2hrestcob}
The restricted $2$-coboundaries for $(\h,\theta)$ are given by pairs $(\varphi,\w)$, where

        $$\begin{cases}
            \varphi(x,y)&=Ax+By+\tilde{C}z\\
            \varphi(x,z)&=Hz\\
            \varphi(y,z)&=Gz,\\
        \end{cases}$$ with $A,B,\tilde{C},D,E,G,H$ belonging to $\F$ and
\begin{itemize}
    \item[$\bullet$] \underline{Case $\theta=0$}: $\w(x)=Ez,~\w(y)=Dz,~\w(z)=0;$ 
  
    \item[$\bullet$] \underline{Case $\theta=z^*$}:
    $\w(x)=Ez,~\w(y)=Dy,~\w(z)=Gx+Hy+Iz.$
\end{itemize}
\end{lem}

With Lemmas \ref{2hrescocy} and \ref{2hrestcob}, we  compute a basis for the second cohomology spaces.

\begin{thm}
We have $\dim_{\F}\left(H^2_{*_2}(\h,0)\right)=3$ and $\dim_{\F}\left(H^2_{*_2}(\h,z^*)\right)=2.$
\begin{itemize}
\item[$\bullet$] A basis for $H^2_{*_2}(\h,0)$ is given by $\{(\varphi_1,
\w_1), (\varphi_2,\w_2), (0,\w_3)\}$, with
$$\varphi_1(y,z)=z;~\varphi_2(x,z)=z;~\w_1(y)=y;~\w_2(x)=x;~\w_3(z)=z.$$
(We only write non-zero images).
\item[$\bullet$] A basis for $H^2_{*_2}(\h,z^*)$ is given by $\{(\varphi_1,
\w_1), (\varphi_2,\w_2)\}$, with
$$\varphi_1(x,y)=x;~\varphi_2(x,y)=y;~
\w_1(y)=y;~\w_2(x)=x.$$
\end{itemize}
\end{thm}

\subsubsection{Restricted Lie-Rinehart structure and restricted deformations}

In this section, we will construct an infinitesimal deformation of a restricted Lie-Rinehart algebra in characteristic $2$, whose underlying Lie algebra is $(\h,z^*)$. Over a field $\F$ of characteristic $2$, there are three non-isomorphic commutative unital associative algebras of dimension $2$, given by
\begin{itemize}
    \item $A_0=\text{Span}_{\F}\{e_1,e_2\},$ with $e_1$ the unit and $e_2e_2=0;$
    \item $A_1=\text{Span}_{\F}\{e_1,e_2\},$ with $e_1$ the unit and $e_2e_2=e_1;$
    \item $A_2=\text{Span}_{\F}\{e_1,e_2\},$ with $e_1$ the unit and $e_2e_2=e_2.$
\end{itemize}

The 2-dimensional classification of unital associative algebras over  over fields of characteristic not $2$algebras reduces to $A_0$ and $A_2$. It turns out that in characteristic 2 the algebra  $A_1$ is not isomorphic to $A_0$ and $A_2$. 
%
\begin{lem}
    Let $L=(\h,z^*)$ and $A=A_1$. Suppose that $A$ acts on $L$ by $a\cdot u=u,~\forall a\in A,~\forall u\in L$. Then, the only restricted Lie-Rinehart structure on the pair $(A,L)$ is given by the anchor map $\rho=0$.
\end{lem}

\begin{proof}
Since $\rho: L\rightarrow\Der(L)$ must be a Lie morphism, we have $\rho(z)=0$. The condition of $A$-linearity gives $\rho(x)(e_2)=\lambda(e_1+e_2)$ and $\rho(y)(e_2)=\mu(e_1+e_2),~\forall \lambda,\mu\in\F$.
Also  $\rho$ must be a restricted morphism. Therefore,  have $0=\rho(x)^2(e_2)=\lambda\rho(x)(e_2)=\lambda^2(e_1+e_2)$. Hence, we have $\lambda=0$. The same computation with $y$ instead of $x$ gives $\mu=0$.
\end{proof}

Let us choose a restricted multiderivation $(m_1,\w_1)$ given by $m_1(x,y)=y,~m_1(x,z)=m_1(y,z)=0$ and $\w_1(x)=x,~\w_1(y)=\w_1(z)=0.$ It is not difficult to see that if we set $\sigma_1(x)(e_2)=\sigma_1(y)(e_2)=e_1+e_2,~\sigma_t(z)(e_2)=0$, then $\sigma_1$ is a symbol map compatible with $(m_1,\w_1)$ that gives rise to an infinitesimal restricted formal deformation. We therefore obtain the deformation
\begin{align*}
    m_t(x,y)&=z+ty,~m_t(x,z)=m_t(y,z)=0;\\
    \w_t(x)&=tx,~\w_t(y)=0,~\w_t(x)=z;\\
    \sigma_t(x)(e_2)&=\sigma_t(y)(e_2)=e_1+e_2,~\sigma_t(z)(e_2)=0.
\end{align*}

It is easy to verify that the quantities $\obs^{(1)}_2$ and $\obs^{(2)}_2$ vanish on the basis of $(\h,z^*)$. Therefore, it is possible to extend the deformation by taking a restricted $2$-cocycle.

\end{document}